\documentclass[reqno]{amsart}

\usepackage{graphicx,pdfsync,amssymb, latexsym,amsfonts,amsbsy,color,subcaption,esint,mathtools,enumerate,nicefrac}

\usepackage{enumitem}
\setitemize{itemsep=0.5em}
\setenumerate{itemsep=0.5em}

\usepackage{hyperref}
\hypersetup{
colorlinks=true,
linkcolor=blue,
filecolor=blue,      
urlcolor=blue,
citecolor=blue,
}

\usepackage[utf8]{inputenc}
\usepackage[T1]{fontenc}

\usepackage{tikz,caption}
\usetikzlibrary{patterns}

\usepackage[top=1.1in, bottom=1in, left=1in, right=1in]{geometry}

\DeclareMathOperator{\Supp }{supp}
\DeclareMathOperator{\Id }{Id} 
 
\DeclareMathOperator{\Lin }{lin} 
\DeclareMathOperator{\Cor }{cor} 
\DeclareMathOperator{\Osc }{osc} 
\DeclareMathOperator{\Rem }{rem} 
\DeclareMathOperator{\Tem }{tem} 
\DeclareMathOperator{\D}{div}
 
\DeclareMathOperator{\dist}{dist}
\DeclareMathOperator{\sgn }{sign}

\newtheorem{theorem}{Theorem}[section]
\newtheorem{lemma}[theorem]{Lemma}
\newtheorem{proposition}[theorem]{Proposition}
\newtheorem{definition}[theorem]{Definition}

\newtheorem{remark}[theorem]{Remark}
\newtheorem{conjecture}[theorem]{Conjecture}

\usepackage{times}
\usepackage{setspace}
\def \TT  {\mathbb{T}} 
\def \RR {\mathbb{R}}  
\def \NN {\mathbb{N}}  
\def \ZZ {\mathbb{Z}}  

\def \ep {\varepsilon}

\def \l {\lambda}

\def \ek {\mathbf{e}_{k}}

\def \bp {\mathbf{ \Phi}}
\def \bw {\mathbf{ W}}

\def \ep {\varepsilon}
\def \p {\partial}

\newcommand{\comment}[1]{}

\numberwithin{equation}{section}

\begin{document}

\title[Nonuniqueness for transport equation at critical space regularity]{Nonuniqueness of weak solutions for the transport equation at critical space regularity}

\author{Alexey Cheskidov}
\address{Department of Mathematics, Statistics and Computer Science,
University of Illinois At Chicago, Chicago, Illinois 60607}
\email{acheskid@uic.edu}

\author{Xiaoyutao Luo}

\address{Department of Mathematics, Duke University, Durham, NC 27708}

\email{xiaoyutao.luo@duke.edu}



\date{\today}

\begin{abstract}
We consider the linear transport equations driven by an incompressible flow in dimensions $d\geq 3$. For divergence-free vector fields $u \in L^1_t W^{1,q}$, the celebrated DiPerna-Lions theory of the renormalized solutions established the uniqueness of the weak solution in the class $L^\infty_t L^p$  when $\frac{1}{p} + \frac{1}{q} \leq 1$.  For such vector fields, we show that in the regime $\frac{1}{p} + \frac{1}{q} > 1$, weak solutions are not unique in the class $ L^1_t L^p$. One crucial ingredient in the proof is the use of both temporal intermittency and oscillation in the convex integration scheme.
\end{abstract}


\maketitle

\section{Introduction}
In this paper, we consider the linear transport equation on the torus $\TT^d$
\begin{equation}\label{eq:the_equation}
\begin{cases}
\partial_t \rho + u \cdot \nabla  \rho  =0 &\\
\rho|_{t=0} =\rho_0,
\end{cases}
\end{equation}
where $\rho : [0,T] \times \TT^d  \to \RR$ is a scalar density function, $u : [0,T] \times \TT^d  \to \RR^d$ is a given vector field. We always assume $u$ is incompressible, i.e.,
$$
\D u =0.
$$

By the linearity of the equation, even for very rough vector fields it is not difficult to prove the existence of weak solutions that solves the equation in the sense of distributions 
\begin{equation}\label{eq:def_weak_solutions}
\int_{\TT^d}\rho_0  \varphi(0,\cdot) \,dx = \int_0^T \int_{\TT^d}\rho (\p_t \varphi + u \cdot \nabla \varphi   ) \,dx dt   \quad \text{for all $\varphi \in C^\infty_c ([0,T) \times \TT^d)$} .
\end{equation}

Our main focus is the uniqueness/nonuniqueness issue for weak solutions to \eqref{eq:the_equation}. More precisely, we investigate whether the DiPerna-Lions uniqueness result is sharp.

\begin{theorem}[DiPerna-Lions~\cite{MR1022305}]
Let $p, q \in [1,\infty]$ and let $u \in L^1(0,T; W^{1, q} (\TT^d)) $ be a divergence-free vector field. For any $\rho_0 \in L^p(\TT^d)$, there exists a unique renormalized solution $\rho \in C( [0,T]; L^{p}(\TT^d) $  to \eqref{eq:the_equation}. Moreover, if
\begin{equation}
\frac{1}{p}  + \frac{1}{q} \leq 1
\end{equation}
then this solution $\rho$ is unique among all weak solutions in class $ L^\infty (0,T; L^p(\TT^d)$.
\end{theorem}

Based on scaling analysis and a close examination of the proof in~\cite{MR1022305}, one can speculate that if
\begin{equation}\label{eq:pq_relationship}
\frac{1}{p} + \frac{1}{q } > 1,
\end{equation}
then the uniqueness may fail. More specifically,

\begin{conjecture}\label{conjecture:main}
Let $p, q \in [1,\infty]$. Let $u \in L^1(0,T; W^{1, q} (\TT^d) ) $ be a divergence-free vector field. 
\begin{enumerate} 
	\item  If $ \frac{1}{p} + \frac{1}{q} \leq 1 $, then there exists a unique weak solution $\rho \in L^\infty( [0,T]; L^{p}(\TT^d) $  to \eqref{eq:the_equation}.
	\item  If $ \frac{1}{p} + \frac{1}{q} > 1 $, then weak solutions in the class $ L^\infty( [0,T]; L^{p}(\TT^d)) $  are not unique.
\end{enumerate}
\end{conjecture}

In this paper, we address the question $(2)$ in Conjecture \ref{conjecture:main} and prove the following.

\begin{theorem}\label{thm:main_thm_short}
Let $d\geq 3$ and $p ,q \in [1,\infty]$ satisfying $p>1$ and \eqref{eq:pq_relationship}. Then there exists a divergence-free vector field $$ u \in L^1 (0,T; W^{1,q}(\TT^d)) \cap L^\infty(0,T; L^{p'}(\TT^d) ),$$ such that  the uniqueness of \eqref{eq:the_equation} fails in the class 
$$ 
\rho \in L^1 (0,T; L^p(\TT^d)).
$$
Moreover, the initial data of such solutions $\rho $ can be attained in the classical sense.
\end{theorem}

This result is proved by the convex integration technique developed over the last decade \cite{MR2600877,MR3090182,MR3866888,1701.08678,MR3884855,MR3898708}, in the spirit of \cite{MR3898708} and using the formulation of \cite{MR3884855}. One key ingredient is the use of both temporal intermittency and oscillations in the convex integration scheme, which is implemented by oscillating stationary building blocks intermittently in time.

\subsection{Background and main results}

It is known that for Lipschitz vector fields, smooth or classical solutions of \eqref{eq:the_equation} can be obtained by solving the ordinary differential equation for the flow map $X: [0,T] \times \TT^d \to \TT^d $
\begin{equation}\label{eq:the_equation_ODE}
\begin{cases}
\p_t X(t,x) = u(t,X(t,x))& \\
X(0,x) = x,
\end{cases}
\end{equation}
and setting $\rho(t,X) =  \rho_0(x)$.
For instance, the wellposedness and uniqueness of \eqref{eq:the_equation} can be deduced from the Cauchy- Lipschitz theory for \eqref{eq:the_equation_ODE}. Moreover, for such vector fields, the inverse flow map $X^{-1}(t)$ solves the transport equation
\begin{equation*} 
\begin{cases}
\p_t X^{-1} + u \cdot \nabla X^{-1} =0& \\
X^{-1}(0) = \Id.
\end{cases}
\end{equation*}

For vector fields that are not necessarily Lipschitz, the link between the PDE \eqref{eq:the_equation} and the ODE \eqref{eq:the_equation_ODE} is less obvious. Even though one can prove the existence of weak solutions fairly easily by the linearity of the equation, the uniqueness issue of \eqref{eq:the_equation} becomes subtler for non-Lipschitz vector fields. The uniqueness class for the density is generally related to the Sobolev/BV regularity of the vector field. The first result in this direction dates back to the celebrated work of DiPerna-Lions \cite{MR1022305} which used the method of renormalization. Since then a lot of effort has been devoted to determining how far the regularity assumption on the vector field can be relaxed. Profound ideas and complex theories, that are beyond the scope of this paper, have been developed, in particular, the notion of \emph{regular Lagrangian flows} introduced by Ambrosio \cite{MR2096794}. We refer to the works \cite{MR2096794,MR2044334,MR2030717,1806.03466,MR1882138, MR3573924,1812.06817,MR2369485}  and the surveys \cite{MR3690748,MR2425023} for regularity/uniqueness results in this direction and for related results of the continuity equation.

Very roughly speaking, there are currently two distinct methods of proving nonuniqueness for \eqref{eq:the_equation}. The first approach is \textit{Lagrangian}, using the degeneration of the flow map to show nonuniqueness at the ODE level; while the second approach is \textit{Eulerian}, using convex integration to prove nonuniqueness directly at the PDE level. 

In regard to the Lagrangian approach, in their original work \cite{MR1022305}, DiPerna and Lions provided a counterexample $u \in W^{1,q}$ with unbounded divergence and a divergence-free counterexample $u \in W^{s,1}$ for all $s< 1$ but $u \not \in W^{1,1} $. Much later, Depauw in \cite{MR2009116} constructed nonuniqueness in the class $\rho \in L^\infty_{t,x}$ for incompressible vector fields $L_{loc}^1 BV$ based on the example in \cite{MR482853}. This type of examples were revisited in \cite{MR3904158,MR2030717,MR3656475} in other contexts. More recently, in \cite{1911.03271}, Drivas, Elgindi, Iyer and Jeong proved nonuniqueness in the class $\rho \in L^\infty_t L^2 $ for $u \in L^1_t C^{1-}$ based on anomalous dissipation and mixing. We should emphasize that the Lagrangian approach is not suited for a construction of a divergence-free example with Sobolev regularity of one full derivative, say $u \in L^1_t W^{1,p}$.

On the Eulerian side, the first nonuniqueness result was obtained by Crippa, Gusev,  Spirito, and Wiedemann in \cite{MR3393180} using the framework of \cite{MR2600877}. However, the vector field $u$ was merely bounded and did not have an associated Lagrangian flow. The first breakthrough result for the Sobolev vector field was obtained by Modena and Sz\'ekelyhidi \cite{MR3884855}. Note that the Sobolev regularity $L^1_t W^{1,p}$ of the vector field implies the uniqueness of a regular Lagrangian flow, see for example \cite{MR3375545}. The contrast between the Lagrangian and Eulerian wellposedness has also been studied in various contexts, see for instance \cite{MR2425023,MR2520506,MR2534294, MR3569243}.

Starting with the groundwork work of Modena and Sz\'ekelyhidi \cite{MR3884855}, the Eulerian nonuniqueness issue of \eqref{eq:the_equation} has drawn a lot of research attention lately. Below are the functional classes where the nonuniqueness has been achieved:

\begin{enumerate}
    \item \cite{MR3884855} (Modena and Sz\'ekelyhidi): $\rho \in C_t L^p  $ when $u \in C_t W^{1,q}\cap C_t L^{p'}$ for $\frac{1}{p}  + \frac{1}{q} > 1 + \frac{1}{d-1}$, $p>1$ and $d\geq 3$. Later in \cite{MR4029736}: extension to the endpoint $p=1$ and $u$ also being continuous.
    
    \item \cite{1902.08521} (Modena and Sattig): $\rho \in C_t L^p  $ when $u \in C_t W^{1,q}\cap C_t L^{p'}$ for $\frac{1}{p}  + \frac{1}{q} > 1 + \frac{1}{d}$ and $d\geq 2$.
    
\item \cite{2003.00539} (Bru\`e, Colombo, and De Lellis): positive\footnote{Note that uniqueness result for positive $\rho$ can go beyond the DiPerna-Lions range, see \cite[Theorem 1.5]{2003.00539}} $\rho \in C_t L^p  $ when $u \in C_t W^{1,q}$ for $\frac{1}{p}  + \frac{1}{q} > 1 + \frac{1}{d}$ , $p>1$ and $d \geq 2$. 
    
\end{enumerate}

In light of the current state, it is then natural to ask whether one can close the gap between the DiPerna-Lions regime $\frac{1}{p} + \frac{1}{q} \leq 1 $ and the Modena-Sattig-Sz\'ekelyhidi regime $\frac{1}{p} + \frac{1}{q} > 1+ \frac{1}{d} $.

In this paper, we address this question and prove nonuniqueness in the full complement of the DiPerna-Lions regime
$$
\frac{1}{p} + \frac{1}{q} > 1
$$
for weak solutions in the class $\rho \in L^1_t L^p $, $p>1$ for dimensions $d\geq 3$.

\begin{theorem}\label{thm:main_thm}
Let $d\geq  3 $ and $p ,q \in [1,\infty]$ satisfying $p>1$ and \eqref{eq:pq_relationship}. For any $\ep>0$ and any time-periodic\footnote{We identify $[0,T]$ with an $1$-dimensional torus.} $ \widetilde{\rho} \in C^\infty(    [0,T] \times \TT^d)$ with constant mean
$$
\fint_{\TT^d} \widetilde{\rho}(t,x) \, dx =\fint_{\TT^d} \widetilde{\rho}(0,x) \, dx \quad \text{for all $t \in [0,T]$},
$$
there exist a vector field $u : [0,T] \times \TT^d \to \RR^d $ and a density $\rho : [0,T] \times \TT^d \to \RR $ such that the following holds.

\begin{enumerate}
	\item $u \in L^1 (0,T; W^{1,q}(\TT^d)) \cap L^\infty (0,T; L^{p'}(\TT^d))$ and  $\rho \in L^1(0,T; L^p(\TT^d))$.

	\item $(\rho, u)$ is a weak solution to \eqref{eq:the_equation} in the sense of \eqref{eq:def_weak_solutions}.
	
	\item The  deviation of $L^p $ norm is small on average: $ \|\rho - \widetilde{\rho} \|_{L^1_t L^p} \leq \ep$.  
	
	\item $  \rho(t)$ is continuous in the sense of distributions and for $t=0, T$, $\rho(t) = \widetilde{\rho}(t) $.
	
	\item The temporal supports satisfy  $\Supp_t \rho \cup \Supp_t u  \subset \Supp_t \widetilde{\rho} $.

\end{enumerate}
\end{theorem}

It is easy to deduce Theorem \ref{thm:main_thm_short} from Theorem \ref{thm:main_thm}.
\begin{proof}[Proof of Theorem \ref{thm:main_thm_short}]
Let $ \rho_0 \in C^\infty_0(\TT^d)$ with $\|\rho_0 \|_p = 1$. We take $\widetilde{\rho}  = \chi(t) \rho_0 (x)$ where $\chi \in C^\infty_c([0,T] )$ is such that $\chi(t) =1 $ if $|t- \frac{T}{2}| \leq \frac{T}{4}$ and $\chi=0$ if $ |t- \frac{T}{2}| \geq \frac{3T}{8}  $. We apply Theorem \ref{thm:main_thm} with $\ep = \frac{T}{1000}$. The obtained solution $\rho$ cannot have a constant $L^p$ norm due to $ \|\rho - \widetilde{\rho} \|_{L^1 L^p} \leq \ep$, and thus is different from the renormalized solution associated with the same vector field $u$ emerging from the same initial data.

Since $\widetilde{\rho}$ vanishes near $t=0$, both the solution $\rho$ and the vector field $u$ also vanish near $t=0$ and hence the initial data is attained in the classical sense.

\end{proof}

\begin{remark}\label{remark:intro}
Several remarks are in order.

\begin{enumerate}

    \item In general, the non-renormalized solution  $\rho  \in L^1_t L^p  $ in Theorem \ref{thm:main_thm} attains its initial data in the sense of distribution instead of in some strong $L^p$ topology as in \cite{MR3884855,MR4029736,1902.08521,2003.00539}. This is due to the artifact of fast temporal oscillations.

    \item For any $k\in \NN$ the vector field $u$  satisfies $\int_0^T \| u (t)\|_{W^{k,\infty}}^r <\infty $  for some small $r>0$ depending on $k $. The ``bad'' part of $u$ concentrates on a small  set\footnote{In fact, the singular set of $u$ is dense, and as a result, there is no local regularity outside the singular set, cf. \cite{MR673830,1809.00600}.} in $[0,T] \times \TT^d$. The density  $\rho $ also satisfies $\int_0^T \| \rho (t)\|_{\infty }^r <\infty $  for some $r >0$.

    \item $L^1_t L^p$ is sharp in terms of the space regularity, but this is achieved at the expense of time regularity by adding temporal intermittency. We discuss this below and in detail in Section \ref{section:temporal_intermittency}. The question of whether the nonuniqueness holds in the class $\rho \in L^\infty_t L^p$ remains open.

    \item It seems possible to also cover the border case $p= 1$ by utilizing the technique in \cite{MR4029736}(see also \cite{MR3374958,1701.08678}).

\end{enumerate}
\end{remark}

\subsection{Continuity-defect equation and the convex integration scheme}

Let us outline the main ideas and strategies of the proof. We follow the framework of \cite{MR3884855} to treat both $\rho$ and $u$ as unknowns and construct a sequence of approximate solutions $(\rho_n, u_n R_n)$ solving the continuity-defect equation
\begin{equation}
\begin{cases}
\p_t \rho_n + u_n \cdot \nabla \rho_n = \D R_n \\
\D u_n =0.
\end{cases}
\end{equation}
The vectors $R_n$ are called the defect fields, which arise naturally when considering weak solutions of \eqref{eq:the_equation}. This framework allows us to use the interplay between the density $\rho_n$ and the vector field $u_n$ as in a nonlinear equation.

The main goal is to design suitable perturbations $\theta_n:= \rho_{n} - \rho_{n-1}$ and $w_n:= u_{n} - u_{n-1}$ such that the defect fields $R_n \to 0$ in an appropriate sense. The most important step is to ensure the oscillation part
\begin{equation}\label{eq:intro_osc_error}
 \D R_{\Osc}:= \D  (\theta_n w_n + R_{n-1} )
\end{equation}
 consists of only high frequencies so that the new defect field $R_n$ is much smaller than $R_{n-1}$. This technique is now considered standard among the experts, and we refer readers to \cite{MR2600877,MR3090182,MR3374958,MR3530360,MR3866888,1610.00676,1812.08734,1812.11311,1907.10436} for more discussion on this technique in other models.

In previous works \cite{MR3884855,MR4029736,1902.08521}, perturbations $(\rho_n, w_n)$ are designed so that \eqref{eq:intro_osc_error} has only high frequencies in space, and the error is canceled point-wise in time. In these works, the defect field $R_n \to 0$ in the norm $L^\infty_t L^1$. In particular, the final solution is homogeneous in time.

In this paper, we use a convex integration scheme that features both spatial and temporal oscillations.  This is done by adding in temporal oscillation when designing $(\rho_n, w_n)$ such that, to the leading order, \eqref{eq:intro_osc_error} can be split into two parts, one with high spatial frequencies, and the other with high temporal frequencies. This idea is implicitly rooted in the work \cite{MR3898708}, but it was not formulated to encode temporal intermittency but rather to cancel a part of the error caused by adding spatial intermittency.

Based on the above discussion, on the technical side, the defect fields $R_n $ shall be measured in $L^1_{t,x}$ instead of $L^\infty_t L^1$. In other words, the defect fields $R_n$ are canceled weakly in space-time, rather than pointwise in time and weakly in space. This relaxation allows us to exploit temporal intermittency and design the perturbations $(\rho_n, w_n)$ with critical space regularity, which we discuss below.

\subsection{Space-time intermittency in the convex integration}

Even though the concept of intermittency and its theoretical studies has been around for many years \cite{MR0495674,MR1428905,MR3152734} in hydrodynamic turbulence, it was only implemented with convex integration very recently in the seminal work \cite{MR3898708} of Buckmaster and Vicol. We can summarize the difficulty as follows. At the heart of its argument, convex integration relies on adding highly oscillatory perturbations to obtain weakly converging solutions. A more intermittent perturbation carries a more diffused Fourier side and introduces more interactions among oscillations. These harmful interactions are difficult to control and cause the iteration scheme to break down. We refer to \cite{MR2600877,MR3090182,MR3254331,MR3374958,MR3530360,MR3866888,1701.08678} for the birth and development of this technique in the fluid dynamics and \cite{MR3898708,MR3951691,1901.07485,MR3884855,1902.08521,1907.10436} and the survey \cite{MR4073888} for discussions on intermittency in convex integration.

To fix ideas, let us denote by $D$  the intermittency dimension (in space), cf.~\cite{MR1428905}. Roughly speaking, the solution is concentrated on a set of dimension $D $ in space. This is related to the development of ``concentration'' in the context of weak solutions, \cite{MR877643,MR882068}.

However, for the transport equation, using only spacial intermittency in a convex integration scheme is not enough to reach the full complement of DiPerna-Lions regime. If the solution $(\rho , u)$ is homogeneous in time, then by the duality $ \rho \in L^\infty_t L^p$ and $u\in L^\infty_t L^{p'} $ imposed by the machinery of convex integration, we can see that
\begin{equation}\label{eq:intro_no_temporal_intermittency}
 u\in L^\infty_t L^{p'} \Rightarrow u\in L^\infty_t W^{1,q} \,\text{  for }\, \frac{1}{p} + \frac{1}{q} > 1 + \frac{1}{d-D}.   
\end{equation}

In other words, the Sobolev regularity $u \in L^\infty W^{1,q}$ must come at the cost of integrability in space if the vector field is homogeneous in time. This simple heuristics works surprisingly well and explains the gap between the DiPerna-Lions regime $\frac{1}{p} + \frac{1}{q} \leq 1 $ and the Modena-Sattig-Sz\'ekelyhidi regime $\frac{1}{p} + \frac{1}{q} > 1+ \frac{1}{d} $ even when spatially fully intermittent $D=0$ building blocks were used in \cite{1902.08521,2003.00539}.

One of the most striking differences between previous schemes and the current one is that intermittency in space plays a very  \emph{little} role. In fact, we use the ``Mikado densities'' and ``Mikado fields'' in \cite{MR3884855} which is not spatially fully intermittent but only has a $d-1$-dimensional concentration. Furthermore, the convex integration scheme goes through as long as the stationary building blocks are not spatially homogeneous.

By contrast, the Sobolev regularity $ u \in L^1_t W^{1,q}$ emerges entirely from the temporal intermittency  in our construction that does not rely on the fundamental heuristics \eqref{eq:intro_no_temporal_intermittency} as previous works. Instead, we take advantage of the duality $ \rho \in L^1_t L^p$ and $u\in L^\infty_t L^{p'} $, which is consistent with the decay of the defect field in $L^1_{t,x}$ rather than $ L^\infty_t L^1_x$ norm.  The temporal intermittency of the vector field $u$ allows us to improve the space regularity as the expense of a worse time regularity, namely
\begin{equation}\label{eq:intro_temporal_intermittency}
u\in L^\infty_t L^{p'} \Rightarrow u \in L^1_t W^{1,q} \,\text{  for }\, \frac{1}{p} + \frac{1}{q} > 1.     
\end{equation}

Indeed, if $u$ is fully intermittent in time, then $L^\infty_t$ to $L^1_t$ embedding gains a full derivative in time. By a dimensional analysis, $ u \in L^1_t W^{1,q} $ can be achieved in \eqref{eq:intro_temporal_intermittency} as long as the associated temporal frequency is comparable to the spacial frequency, since $q < p'$ by $ \frac{1}{p} + \frac{1}{q} > 1$. Note that this approach requires a sharper estimate of the error involving the time derivative since the temporal frequencies become as large as spacial frequencies. We also emphasize that the heuristics \eqref{eq:intro_temporal_intermittency} encodes no information on spacial intermittency, which is fundamentally different from the heuristics \eqref{eq:intro_no_temporal_intermittency}. In fact, the convex integration scheme works in a wide range of concentration and oscillation parameters. We refer to Section \ref{section:prop_proof1} for the specific choice of parameters and Lemma \ref{lemma:tem_error} for the sharp estimate of the temporal error.

\subsection{Temporal intermittency via oscillating stationary solutions}
We will now describe the implementation of temporal intermittency needed to reach the optimal spatial regularity. This is achieved via oscillating stationary building blocks intermittently in time.

Current convex integration schemes employ spatially intermittent building blocks that are either not stationary \cite{1902.08521,2003.00539}, or stationary \cite{MR3884855} but only suitable for $d \geq 3$. Even though theoretically it seems to be possible to achieve temporal intermittency using non-stationary building blocks \cite{1902.08521,2003.00539} in $d \geq 2$, such an approach, if possible, would be less intuitive and significantly more complicated, and it is not clear to the authors that one can reach the same nonuniqueness regime as in the current paper. Our approach adheres closer to the original idea of adding space-time oscillations to stationary solutions implemented in the pioneering work \cite{MR2600877} that introduced the convex integration technique to fluid dynamics for the first time.

To perform convex integration on $\TT^d$, $d\geq 3$ we use the stationary Mikado density $\bp_k$ and Mikado flow $\bw_k$ on $\TT^d$, first introduced in \cite{MR3884855}. With stationary building blocks $ (\bp_k, \bw_k)$ at hand, we implement the temporal intermittency as follows. On one hand, we use temporal oscillations to relax the convex integration procedure from pointwise to weak in time. Given a solution $(\rho, u ,R)$ of the continuity-defect equation, we design a perturbation $(\theta, w)$ so that, to the leading order, it produces a high-high to low cascade in space-time that balances the old defect field $R$ in the sense that
$$
\D( \theta w + R) = \text{High Spacial Freq. Term} + \text{High Temporal Freq. Term}+  \text{Lower Order Terms}.
$$
The terms with high temporal frequencies will be further balanced by the time derivative of a small corrector, similar to \cite{MR3898708}, while the other terms can be easily handled by standard methods. On the other hand, the relaxation of convex integration to be done weakly in time allows us to add temporal intermittency in the perturbations $(\theta, w)$. The key is to ensure that $(\theta , w)$ is almost fully intermittent in time, which determines the regularity of the final solution $\rho \in L^1_t L^p $ and $u \in L^\infty_t L^{p'} \cap L^1_t W^{1,q}$.

To summarize, in the proposed convex integration scheme, the perturbations consist of space-time intermittent oscillatory building blocks. The temporal intermittency is used to achieve the optimal range in \eqref{eq:intro_temporal_intermittency}, whereas the temporal oscillation allows us to cancel the defect fields on average in space-time, consistent with the decay of $R_n$ in the norm $L^1_{t,x}$. We refer to Section \ref{section:temporal_intermittency} for more details.

\subsection{Organization of the paper}
The rest of the paper is organized as follows.
\begin{itemize}
    \item We introduce the notations and many technical tools used throughout the paper in Section \ref{section:pre}.
    \item Section \ref{section:proof} is devoted to the proof of Theorem \ref{thm:main_thm} by assuming the main proposition, Proposition \ref{prop:main_prop}.

    \item In Section \ref{section:Mikado} we recall the periodic stationary solutions $(\bp_k, \bw_k)$ on $\TT^d$, $d \geq 3$ in \cite{MR3884855}. These pairs $(\bp_k, \bw_k)$ will be the main building blocks in space of the convex integration scheme.

    \item Section \ref{section:temporal_intermittency} is a detailed explanation for the use of temporal oscillation and intermitetncy in the convex integration scheme. In particular, we will define the temporal oscillators $\widetilde{g}_\kappa, {g}_\kappa$ that we use to oscillate the building blocks $(\bp_k, \bw_k)$ in time.

    \item Section \ref{section:prop_proof1}, \ref{section:prop_proof2}, \ref{section:prop_proof3} constitute the proof of Proposition \ref{prop:main_prop}:
    \begin{itemize}
        \item[--]  In Section \ref{section:prop_proof1} we first define the perturbation density $\rho$ and vector field $w$ using the building blocks $(\bp_k, \bw_k)$. And then the new defect field $R$ is derived from the perturbations $\theta$ and $w$, which is the core of our convex integration scheme.
        
        \item[--]  The estimates for the perturbations $\rho$ and $w$ are done in Section \ref{section:prop_proof2}. Then we conclude the proof of the perturbation part of Proposition \ref{prop:main_prop}.
        \item[--]  The new defect field $R$ is estimated is Section \ref{section:prop_proof3}. The rest of the proof of Proposition \ref{prop:main_prop} will be completed in the end.
    \end{itemize}
\end{itemize}
\section{Preliminaries}\label{section:pre}
The purpose of this section to collect the technical tools that will be used throughout the paper. We keep this section relatively concise so that we are not distracted from the main goal of proving the nonuniqueness result.

\subsection{Notations}
Throughout the manuscript, we use the following notations.
\begin{itemize}
\item $\TT^d  = \RR^d / \ZZ^d$ is the $d$-dimensional torus. For any function $f: \TT^d \to \RR$ we denote by $f(\sigma \cdot)$ the $ \sigma^{-1} \TT^d$-periodic function $f(\sigma x)$.

\item For any $p \in [1,\infty]$, its H\"older dual is denoted as $p'$. Throughout the paper, $p$ is fixed as in Theorem \ref{thm:main_thm}. We will use $r$ for general $L^r$ norm.

\item For any $1\leq r \leq \infty$, the Lebesgue space is denoted by $L^r$. For any $f \in L^1(\TT^d) $, its spacial average is
$$
\fint_{\TT^d} f \,dx= \int_{\TT^d} f\,dx.
$$ 

\item For any function $f:[0,T] \times \TT^d \to \RR $, denote by $\| f(t) \|_r $ the Lebesgue norm on $\TT^d$ (in space only) at a fixed time $t$. If the norm is taken in space-time, we use $\|f \|_{L^r_{t,x}} $.

\item The space $C^\infty_0(\TT^d)$ is the set of periodic smooth functions with zero mean, and $C^\infty_c(\RR^d)$ is the space of smooth functions with compact support in $\RR^d$. 

\item We often use the same notations for scalar functions and vector functions. Sometimes we use $ C^\infty_0(\TT^d, \RR^d )$ for the set of periodic smooth vector fields with zero mean.

 \item We use $\nabla$ to indicate full differentiation in space only, and space-time gradient is denoted by $ \nabla_{t,x}$. Also, $\p_t$ is the partial derivative in the time variable.

\item For any Banach space $X$, the Banach space $L^r(0,T;X)$ is equipped with the norm
$$
\Big(  \int_{0}^T \| \cdot  \|_X^r \, dt \Big)^\frac{1}{r},
$$
and we often use the short notations $L^p_t X$ and $\| \cdot \|_{L^r_t X}$.

\item We write $X \lesssim Y$ if there exists a constant $C>0$ independent of $X$ and $Y$ such that $X \leq C Y $. If the constant $C$ depends on quantities $a_1,a_2,\dots,a_n$ we will write $X \lesssim_{a_1,\dots,a_n}$ or $X \leq C_{a_1,\dots ,a_n} Y $

\end{itemize}

\subsection{Antidivergence operators $\mathcal{R}$ and $\mathcal{B} $ on $\TT^d$}
We will use the standard antidivergence operator $\Delta^{-1} \nabla$   on $\TT^d$, which will be denoted by $\mathcal{R}$.

It is well known that for any $f \in C^\infty  (\TT^d)$ there exist a  unique $ u \in C^\infty_0 (\TT^d)$ such that 
$$
\Delta u = f -\fint f.
$$ 
For any smooth scalar function $f \in C^\infty(\TT^d) $, the standard anti-divergence operator $ \mathcal{R} : C^\infty(\TT^d)  \to C^\infty_0(\TT^d, \RR^d ) $ can be defined as
$$
\mathcal{R} f:=  \Delta^{-1} \nabla f ,
$$
which satisfies
$$
\D ( \mathcal{R}  f )  = f  -\fint_{\TT^d}f\quad \text{for all $ f \in C^\infty(\TT^d)$},
$$
and
$$
\| \mathcal{R} ( \D   u ) \|_r  \lesssim \| u \|_r   \quad \text{for all $ u \in C^\infty(\TT^d,\RR^d)$ and $1<  r <\infty$}.
$$

The next result, which says that $ \mathcal{R}$ is bounded on all Sobolev spaces $W^{k,p}(\TT^d)$, is classical, see for instance \cite[Lemma 2.2]{MR3884855} for a proof.

\begin{lemma}\label{lemma:antidivergence_bounded}
Let $d \geq 2$. For every $m \in \NN$ and $ r \in [1,\infty]$, the antidivergence operator $ \mathcal{R}$ is bounded on $ W^{m,r}(\TT^d)$ for any $m \in \NN$:
\begin{equation}
\|\mathcal{R} f \|_{W^{m,r}} \lesssim \|f \|_{W^{m,r }}.
\end{equation}
\end{lemma}

Throughout the paper, we use heavily the following fact about $\mathcal{R}$.
$$
\mathcal{R}f(\sigma \cdot ) = \sigma^{-1} \mathcal{R}f  \quad \text{for any $f \in C^\infty_0(\TT^d)$ and any positive $\sigma \in \NN$.}
$$

We will also use its bilinear counterpart $ \mathcal{B}: C^\infty(\TT^d) \times C^\infty(\TT^d) \to C^\infty (\TT^d,\RR^d) $ defined by
$$
\mathcal{B}(a,f) : = a \mathcal{R} f  - \mathcal{R}( \nabla a  \cdot \mathcal{R} f).
$$

This bilinear version $\mathcal{B}$ has the additional advantage of gaining derivative from $f$ when $f$ has zero mean and a very small period. See also higher order variants of $\mathcal{B}$ in \cite{1902.08521}.

It is easy to see that $\mathcal{B}$ is a left-inverse of the divergence,
\begin{equation}\label{eq:bilinear_B_identity}
\D (\mathcal{B}(a,f)  ) = af -\fint_{\TT^d} af \, dx \quad \text{provided that $f \in C^\infty_0(\TT^d)$, }
\end{equation}
which can be proved easily using integration by parts.
The following estimate is a direct consequence of Lemma \ref{lemma:antidivergence_bounded}.
\begin{lemma}\label{lemma:cheapbound_B}
Let $d \geq 2$ and $1\leq r  \leq \infty$. Then for any $a,f \in C^\infty(\TT^d)$
\begin{align*}
\| \mathcal{B}(a,f)\|_r \lesssim \| a\|_{C^1} \|\mathcal{R}  f\|_r .
\end{align*}
\end{lemma}
\begin{proof}
This follows from H\"older's inequality and Lemma \ref{lemma:antidivergence_bounded}.
\end{proof}

\begin{remark}
The assumption on $f$ in Lemma \ref{lemma:cheapbound_B} can be relaxed to $f \in L^r(\TT^d)$. 
\end{remark}

\subsection{Improved H\"older's inequality on $\TT^d$}
We recall the following result due to Modena and Sz\'ekelyhidi \cite[Lemma 2.1]{MR3884855}, which extends the first type of such result \cite[Lemma 3.7]{MR3898708}.

\begin{lemma}\label{lemma:improved_Holder}
Let $\sigma \in \NN$ and $a,f :\TT^d \to \RR$ be smooth functions. Then for every $r \in [1,\infty]$,
\begin{equation}
\Big|   \|a f(\sigma \cdot ) \|_{r }  - \|a \|_{r} \| f \|_{r } \Big|\lesssim \sigma^{-\frac{1}{r}} \| a\|_{C^1} \| f \|_{ r }.
\end{equation}
\end{lemma}

This result allows us to achieve sharp $L^r $ estimates when estimating the perturbations in Section \ref{section:prop_proof2}. Note that the error term on the right-hand side can be made arbitrarily small by increasing the oscillation $ \sigma$.

\subsection{Mean values and oscillations}
We use the following Riemann-Lebesgue type lemma.

\begin{lemma}\label{lemma:mean_values_oscillation}
Let $\sigma \in \NN $ and $a , f : \TT^d  \to \RR$ be smooth functions such that $f \in C_0^\infty ( \TT^d)$. Then for all even $n \geq 0$
\begin{equation}
\Big|     \fint_{\TT^d } a(x) f (\sigma x) \, dx \Big| \lesssim_n \sigma^{-n} \|a \|_{C^n} \| f \|_{2}.
\end{equation}
\end{lemma}
\begin{proof}
Since $f$ has zero mean, by repeatedly integrating by parts we deduce that
$$
\fint_{\TT^d } a(x) f (\sigma x) \, dx = \sigma^{-n} \fint_{\TT^d } \Delta^{n/2}  a  \Delta^{-n/2} f (\sigma \cdot ) \, dx.
$$
On one hand, we have
$$
\|\Delta^{n/2}  a \|_{L^2(\TT^d)} \lesssim \| a \|_{C^{n}(\TT^d)} .
$$
On the other hand, since $f$ is zero-mean, by the Plancherel theorem
\[
\|\Delta^{-n/2} f \|_{L^2(\TT^d)} \lesssim \|  f   \|_{L^2(\TT^d)}.
\]

Thus for any even $n  $ we have
$$
\Big|     \fint_{\TT^d } a(x) f (\sigma x) \, dx \Big| \lesssim_n \sigma^{-n} \|a \|_{C^n} \| f \|_{2}.
$$
\end{proof}

\section{The main proposition and proof of Theorem \ref{thm:main_thm}}\label{section:proof}

\subsection{Time-periodic continuity-defect equation}
We follow the framework of \cite{MR3884855} to obtain approximate solutions to the transport equation by solving the continuity-defect equation
\begin{equation}\label{eq:defect_equation}
\begin{cases}
\p_t \rho + \D(\rho u) = \D R\\
\D  u =0,
\end{cases}
\end{equation}
where $R: [0,T] \times \TT^d \to \RR^d$ is called the defect field. In what follows, $(\rho, u, R)$ will denote a solution to \eqref{eq:defect_equation}.

Throughout the paper, we assume $T=1$ and identify the time interval $[0,1]$ with an $1$-dimensional torus. As a result, we will only consider smooth solutions $(\rho, u, R) $ to \eqref{eq:defect_equation} that are time-periodic as well, namely
$$
\rho(t+k)= \rho(t), \quad u(t+k)= u(t), \quad R(t+k)= R(t) \quad \text{for any $k \in \ZZ$}.
$$

For any $r>0$, let
$$
I_r:= [r,1-r].
$$
We now state the main proposition of the paper and use it to prove Theorem \ref{thm:main_thm}.
\begin{proposition}\label{prop:main_prop}
Let $d \geq  3$ and $p ,q \in [1,\infty]$ satisfying $p>1$ and \eqref{eq:pq_relationship}. There exist a universal constant $M>0$ and a large integer $N\in \NN$ such that the following holds. 

Suppose $(\rho, u, R) $ is a smooth solution of \eqref{eq:defect_equation} on $[0,1]$. Then for any $\delta,\nu >0$, there exists another smooth solution $(\rho_1, u_1, R_1)$ of \eqref{eq:defect_equation} on $[0,1]$ such that the density perturbation $\theta: = \rho_1 -\rho$ and the vector field perturbation $ w = u_1 -u$ verify the estimates
\begin{align}
 \|\theta \|_{L^1_t L^p} & \leq  \nu M \|R \|_{L^1_{t,x}}^{1/p}, \label{eq:main_prop_0}\\
\| w \|_{L^\infty_t L^{p'}} &\leq \nu^{-1} M \|R \|_{L^1_{t,x}}^{1/p'},  \label{eq:main_prop_1}\\
 \| w & \|_{L^1_t W^{1, q}}  \leq \delta . \label{eq:main_prop_2}
\end{align}
In addition, the density perturbation $\theta$ has zero spacial mean and satisfies
\begin{align}
\Big|   \int_{\TT^d } \theta (t,x) \varphi(x)  \, d x \Big|  & \leq \delta  \| \varphi \|_{C^N}, \quad \text{for any $t \in [0,1]$ and any $\varphi \in C^\infty(\TT^d)$} ,  \label{eq:main_prop_3}\\
\Supp \theta \subset   I_r \times    \TT^d &\quad \text{for some $r >0$}, \; \text{and} \;
\Supp_t \theta \cup \Supp_t w \subset \Supp_t R
. \label{eq:main_prop_4}
\end{align} 

Moreover, the new defect field $R_1$ satisfies
\begin{align}
\|R_1 \|_{L^1_{t,x}} \leq \delta. \label{eq:main_prop_5}
\end{align}
\end{proposition}

\subsection{Proof of Theorem \ref{thm:main_thm} }
\begin{proof}\label{subsec:proof_main_thm}
We assume $T=1$ without loss of generality. We will construct a sequence $(\rho_n, u_n , R_n) $, $n=1,2\dots$ of solutions to \eqref{eq:defect_equation} as follows. 
For $n=1$, we set 
\begin{align*}
\rho_1(t) &:= \widetilde{\rho},  \\ 
u_1(t) &:= 0, \\
R_1(t) &:=  \mathcal{R} \big( \p_t \widetilde{\rho} \big).
\end{align*}
Then $ (\rho_1, u_1 , R_1)  $ solves \eqref{eq:defect_equation} trivially by the constant mean assumption on $\widetilde{\rho} $. 

Next, we apply Proposition \ref{prop:main_prop} inductively to obtain $ (\rho_n, u_n , R_n)$ for $n=2,3\dots$ as follows. Let 
$$
\nu =  \frac{ \ep }{2 M} \| R_1 \|_{L^1_{t,x}}^{-\frac{1}{p}}, \quad \delta_n : = 2^{-p (n-1)} \| R_1 \|_{L^1_{t,x}}, 
$$
where we note that $1 <p,p'<\infty$ by the assumptions on $p,q$.

Given $(\rho_n, u_n,R_n) $, we apply Proposition \ref{prop:main_prop} with parameters $ \nu$ and $\delta_n$ to obtain a new triple $(\rho_{n+1}, u_{n+1},R_{n+1}) $. Then the perturbations $\theta_n : = \rho_{n+1} - \rho_{n}$ and $w_n : = u_{n+1} - u_{n}$ verify
\begin{align*}
 \|\theta_n \|_{L^1_t L^p} \leq M \nu \delta_n^{\frac{1}{p}}, 
 \quad \quad \| w_n \|_{L^\infty_t L^{p'}} \leq M \nu^{-1} \delta_n^{\frac{1}{p'}}, 
 \end{align*}
 and
 \begin{align*}
 \|w_n \|_{L^1_t W^{1, q}} &\leq  \delta_n, \\
 \Supp \theta_n \subset I_{r_n} \times & \TT^d \quad \text{for some $r_n >0$},
 \end{align*}
for all $n =1,2\dots$.
So there exists $(\rho, u) \in L^1_t L^p \times L^\infty_t L^{p'}$ such that
\begin{align}
 \rho_n  &\longrightarrow  \rho \quad \text{in } {L^1_t L^p},\\
 u_n &\longrightarrow  u  \quad \text{in } {L^\infty_t L^{p'} \cap L^1_t W^{1, q}} .
\end{align}
It is standard to prove $(\rho, u) $ is a weak solution to \eqref{eq:the_equation} since 
$$
\rho_n u_n  \longrightarrow  \rho u \quad \text{in } {L^1_{t,x}}.
$$
Moreover,
$$
\|\rho - \widetilde{\rho} \|_{L^1_t L^p} \leq \sum_{n\geq 1} \|\theta_n \|_{L^1_t L^p}  \leq \sum_{n \geq 1} \ep 2^{-n} \leq \ep.
$$

To show that $\rho(t)$  is continuous in the sense of distributions, let $  \varphi \in C_c^\infty(\TT^d)$. 
It follows that
\begin{align*}
\langle \rho(t)  - \rho(s) , \varphi \rangle & \leq \Big|  \langle \rho(t)  - \rho_n(t) , \varphi \rangle \Big| +\Big|  \langle \rho_n(t)  - \rho_n(s) , \varphi \rangle \Big|\\
& \quad + \Big|  \langle \rho_n(s)  - \rho(s) , \varphi \rangle  \Big|.
\end{align*}
Since by \eqref{eq:main_prop_2}
$$
\Big|  \langle \rho(t)  - \rho_n(t) , \varphi \rangle \Big| \leq \sum_{k \geq n+1} \delta_n \|\varphi \|_{C^N} \quad  \text{for all $ t \in [0,1]$},
$$
the continuity of $\rho$ in distribution follows from the smoothness of $\rho_n$.

The claim that $ \rho(t)  = \widetilde{\rho}( t)$ for $t=0,1$ follows from the fact that $ \rho_n(0) = \rho(0)$ and $ \rho_n(1) = \rho(1) $ for all $n$ since 
$$
\Supp \theta_n    \subset I_{r_n} \times \TT^d.
$$ 

Finally, the claim that $ \Supp_t \rho \cup \Supp_t u  \subset \Supp_t \widetilde{\rho} $ follows from the fact that $\Supp_t R_1 \subset  \Supp_t \widetilde{\rho}$.
\end{proof}


\section{Stationary Mikado density and Mikado fields}\label{section:Mikado}

In this section, we recall the construction of the stationary Mikado density and Mikado fields introduced by Modena and Sz\'ekelyhidi in \cite{MR3884855} with its roots dating back to \cite{MR3614753}, which will be used as the building blocks in space in the convex integration scheme.

Let $d \geq 3$ be the spacial dimension. We fix a vector field $\Omega \in C_c^\infty(\RR^{d-1} ,\RR^{d-1})  $ such that
$$
\Supp \Omega \in (0,1)^{d-1},
$$
and denote by $\phi \in C_c^\infty(\RR^{d-1})$ the solution
$$
\D \Omega = \phi.
$$

The vector field $ \Omega$  is also normalized such that
$$
 \quad\int_{\RR^{d-1}} \phi^2 = 1.
$$

Throughout this section, for $\mu >0$, we denote $\phi^{\mu} = \psi(\mu x)$ and $\Omega^{\mu} = \Omega(\mu x)$.

\subsection{Non-periodic Mikado densities $\Phi_k $ and Mikado fields $W_k$}

For each $k = 1,\dots, d$, we define
\begin{equation}
\Phi_k (x_1,\dots, x_d) = \mu^{\frac{d-1}{ p }} \phi^{\mu} (x_1,\dots,x_{k-1},x_{k+1},\dots, x_d),
\end{equation}
and
\begin{equation}
W_k (x_1,\dots, x_d) = \mu^{\frac{d-1}{ p' }} \phi^{\mu} (x_1,\dots,x_{k-1},x_{k+1},\dots, x_d) \ek,
\end{equation}
where $\ek $ is the $k$-th standard Euclidean basis and the exponents $p,p' \in (1,\infty)$ are as in Theorem \ref{thm:main_thm}.

Since $\Phi_k $ and $W_k$ effectively depend only on $d-1$ coordinates, we have the following.

\begin{theorem}[Exact stationary solution $(\Phi_k, W_k )$]\label{thm:WPhi_exact_stationary}
Let $d \geq 3$ and $\mu>0$. The density $\Phi_k : \RR^d \to \RR $ and the vector field  $W_k  : \RR^d \to \RR^d $, $k=1,\dots,d$  verify the following.

\begin{enumerate}
\item  $\Phi_k,W_k \in C^\infty(\RR^d)$. Both $\Phi_k$ and $W_k$ have zero mean in the unit cube
$$
\int_{ (0,1)^d}  \Phi_k  = \int_{ (0,1)^d } W_k   =0.
$$

\item The Mikado fields are divergence-free:
$$
\D W_k  = 0,
$$
and the pair $(\Phi_k  ,W_k  )$ solves the stationary transport equation
$$
\D( \Phi_k W_k )=  W_k \cdot \nabla \Phi_k =0.
$$

\item There holds 
\begin{equation}
\int_{ (0,1 )^d  }  \Phi_k  W_k  = \ek.
\end{equation}
 
\item There exists vector potentials $ \Omega_k \in C^\infty(\RR^d,\RR^d)$ such that
$$
\D  \Omega_k = \Phi_k 
$$
and 
$$
\| \Omega_k\|_{L^r((0,1)^d)} \lesssim  \mu^{-1 + \frac{d-1}{p } - \frac{d-1}{r}}.
$$

\end{enumerate}

\end{theorem}
\begin{proof}
The first three properties are standard and follow directly from the definition.

For the last property, let us define
$$
\Omega_k : = \mu^{-1+\frac{d-1}{p}}  \Omega^\mu   ( x_1,\dots,x_{k-1},x_{k+1},\dots,x_d).
$$
Then by the definition of $\phi $ and $\Omega$, we have
$$
\D(\Omega_k ) = \Phi_k.
$$

The estimate for $\Omega_k$ follows immediately
\begin{align*}
\| \Omega_k\|_{L^r((0,1)^d)} &\leq \mu^{-1+\frac{d-1}{p}} \|  \Omega^\mu \|_{L^{r}(\RR^{d-1})} \\
&\lesssim   \mu^{-1 + \frac{d-1}{p } - \frac{d-1}{r}}.
\end{align*}
 
\end{proof}

\subsection{Geometric setup and periodization}
Next, we use the obtained non-periodic solutions $ ((\Phi_k , W_k ) $ to generate a family of $d$ pairs $ (\bp_k, \bw_k )$ by translation and periodization. The goal is to make sure $ (\bp_k, \bw_k )$ centered at disjoint line in $(0,1)^d$ that are parallel to the Euclidean basis $ \ek$.  

We choose a collection of distinct points $p_{ k} \in [1/4,3/4]^d $ for $k = 1,\dots, d$ and a number $\ep_0 >0$ such that
$$
\bigcup_{k} B_{\ep_0}(p_k) \subset [0,1]^d,
$$
and
\begin{equation}\label{eq:ep_0_pi_pj}
\dist(l_k,l_{k'})\geq  \ep_0 \quad \text{if  }  k\neq k' . 
\end{equation}
where $l_k \subset \RR^d$ is the line passing through $p_k$ with direction $k$. The lines $l_k$ will be the centers of our solutions $ (\bp_k, \bw_k ) $. 
We then choose   translations $ \mathfrak{T}_k :\TT^d \to \TT^d $ for $k=1,\dots, d$,
$$
\mathfrak{T}_k x := x + p_k.
$$

Now we are ready to introduce the periodic solution $(\bp_k , \bw_k)$ as the $1$-periodic extension of $(\Phi_k , W_k)$.

\begin{definition}[Periodic solutions]\label{def:periodic_Phi_k_W_k}

Let $d\geq3$. Define periodic density $\mathbf{ \Phi} : \TT^d \to \RR^d$ and periodic vector fields $\mathbf{W}_{k } :\TT^d \to \RR^d $ by
\begin{align*}
\mathbf{ \Phi}_k (\mathfrak{T}_k x)& =  \sum_{\substack{j \in \ZZ^d\\ j_k=0} }  \Phi_k (x+j) , \\
\mathbf{W}_k (\mathfrak{T}_k x)  &=  \sum_{\substack{j \in \ZZ^d\\ j_k=0} } W_k(x+j),
\end{align*}
and the periodic potential $\mathbf{\Omega}_k : \TT^d \to \RR^d$ by
$$
\mathbf{\Omega}_k (\mathfrak{T}_k x) : =  \sum_{\substack{j \in \ZZ^d\\ j_k=0} }      \Omega_k  (x+j).
$$
\end{definition}

From the definition we immediately obtain Sobolev estimates for $( \bp_k,\bw_k)$.
\begin{proposition}\label{prop:bounds_periodic_phi_w}
Let $d\geq3$. For any $ \mu  > 0$, the following estimates hold for any $1\leq r\leq \infty  $:
\begin{align*}
 \mu^{-m}\big\| \nabla^m     \bp_k  \big\|_{L^r(\TT^d)}  & \lesssim_{ m}     \mu^{  \frac{d-1}{p}  -\frac{d-1}{r}} , \quad  m \in \NN ,  \\
 \mu^{-m}\big\| \nabla^m     \mathbf{\Omega}_k  \big\|_{L^r(\TT^d)}&  \lesssim_{ m}    \mu^{ -1+   \frac{d-1}{p }  -\frac{d-1}{r}}    , \quad  m \in \NN \\
\mu^{-m}\big\| \nabla^m     \mathbf{W}_k  \big\|_{L^r(\TT^d)}&  \lesssim_{ m}    \mu^{  \frac{d-1}{p'}  -\frac{d-1}{r}}    , \quad  m \in \NN .
\end{align*}
\end{proposition}
\begin{proof}
These estimates follow immediately from rescaling.
\end{proof}

\begin{theorem}[Stationary periodic solution $( \bp_k,\bw_k)$]\label{thm:peridoc_Phi_k_W_k}
Let $d\geq3$ and $\mu \geq 2 \ep_0^{-1}$. The periodic solutions $\bp_k, \bw_k \in C^\infty_0(\TT^d) $  verify the following.
\begin{enumerate}
    \item The vector field $\bw_k$ is divergence-free,
$$
\D \bw_k = 0,
$$
and the pair  $( \bp_k,\bw_k)$ solves the stationary transport equation
\begin{equation}
\D(\bp_{k}   \mathbf{W}_{k} )  =    \mathbf{W}_{k} \cdot \nabla \bp_{k}   = 0.
\end{equation}

\item The density $\mathbf{ \Phi}_k $ is the divergence of the potential $\mathbf{\Omega}_k$,
\begin{align}\label{eq:Div_Omega_k_close_Phi_k}
  \D \mathbf{\Omega}_k  =  \mathbf{ \Phi}_k .
\end{align}

\item  There holds 
\begin{equation}
\int_{ \TT^d  }  \bp_k  \bw_k  = \ek
\end{equation} and if $k \neq k'$, then
\begin{align}\label{eq:far_field_interactions}
\Supp \bp_k  \cap \Supp \bw_{k'}  = \emptyset.
\end{align}
\end{enumerate}
\begin{proof}
The first property follows directly from Theorem \ref{thm:WPhi_exact_stationary} while the second property follows from the definition.

The last property follows from the facts that $\Supp \Phi \subset (0,1)^{d-1}$ and $\dist(l_i,l_j) \geq \ep_0$ if $i \neq j$.
 
\end{proof}

 \end{theorem}

\section{Temporal intermittency and oscillation}\label{section:temporal_intermittency}
Here we introduce one of the key ingredients of this paper, the use of both temporal intermittency and oscillation. This allows us to kill the previous defect field in a space-time average fashion instead of point-wise in time.

For convenience, we will treat the time interval $[0,1]$ as an $1$-dimensional torus $\TT$. In what follows we always write $[0,1]$ as an interval in time to distinguish it from the periodicity in space.

\subsection{Limitations of the previous schemes}
We start by discussing how the Sobolev regularity was obtained in previous convex integration schemes \cite{MR3884855,MR4029736,1902.08521}.  Assume that we have $(\rho, u ,R)$ the solution to the defect equation
$$
\p_t \rho + u \cdot \nabla \rho = \D R,
$$
the goal is to design suitable perturbations $(\theta , w)$ such that $(\rho+ \theta , u+ w)$ is a new solution to the defect equation with a smaller defect field $R_1$.

Typical in the convex integration scheme, the principle part of the perturbation $(\theta , w)$ takes the form
\begin{equation}
\theta = \sum_k a_k \bp_k, \qquad w = \sum_k b_k \bw_k.
\end{equation}
The coefficients $a_k,b_k$, depending on the previous defect field $R$, are chosen such that the leading order high-high to low interaction balance the defect field $R$
$$
\sum_k a_k b_k\fint_{\TT^d} \bp_k \bw_k \, dx + R \sim 0.
$$

Heuristically, without temporal intermittency, the duality given by the perturbation is
\begin{equation}\label{eq:time_intermittency_duality_1}
\theta \in  L^\infty_t L^p \quad w \in L^\infty_t  L^{p'}.
\end{equation}
To require Sobolev regularity of the vector field $w$, one has to trade in some integrability in space to obtain
\begin{equation}\label{eq:time_intermittency_duality_2}
w \in L^\infty_t W^{1,q} \quad \text{for q such that  } \frac{1}{p} + \frac{1}{q}  > 1 + \frac{1}{d-D}    
\end{equation}
where $D$ is the intermittency dimension of $(\bp_k,\bw_k)$. This has been done in \cite{MR3884855,MR4029736} for $D=1$ and in \cite{1902.08521,2003.00539} for $D=0$.
So the approach of using only spacial intermittency results in the nonuniqueness in the range
$$
\frac{1}{p} + \frac{1}{q} > 1 + \frac{1}{d}.
$$

\subsection{Convex integration with space-time intermittency and oscillation}
Our approach is to add in temporal intermittency and oscillations to the perturbation $(\theta, w)$,
\begin{equation}
\theta = \widetilde{g}_{\kappa} \sum_k a_k \bp_k, \quad w = g_{\kappa} \sum_k b_k \bw_k,
\end{equation}
where $ \widetilde{g}_{\kappa} ,  g_{\kappa} : [0,1] \to \RR$ are intermittent functions in time with oscillations. 

By imposing the duality 
$$ 
\int_{[0,1]} \widetilde{g}_{\kappa}    g_{\kappa} \, dt = 1,
$$ 
we anticipate that the defect field is canceled weakly in space-time
\begin{equation}\label{eq:heuristics_convex_spacetime}
\sum_k a_k b_k \fint_{[0,1] \times \TT^d }  \widetilde{g}_{\kappa}  g_{\kappa} \bp_k \bw_k \, dx dt + R \sim 0  .
\end{equation}

This would allow us to obtain additional regularity in space at the expense of regularity in time, which means that $ \widetilde{g}_{\kappa}$ and $g_\kappa$ have to have different scalings for different $L^p$ norms.
 Indeed, with intermittency in time, we impose the duality between $ \theta$ and $w$ to be
\begin{equation}\label{eq:time_intermittency_duality_3}
\theta \in  L^1_t L^p, \qquad w \in L^\infty_t L^{p'},
\end{equation}
which is also consistent with the ansatz \eqref{eq:heuristics_convex_spacetime}.

The hope is that with enough temporal intermittency, we get
\begin{align}\label{eq:temporal_intermittency_gain}
w \in L^\infty_t L^{p'} \Rightarrow  w \in L^1_t W^{1,q} \quad \text{for $\frac{1}{p} + \frac{1}{q} > 1$}.
\end{align}
Note that temporal intermittency is the key difference between \eqref{eq:time_intermittency_duality_2} and \eqref{eq:temporal_intermittency_gain}.

After performing convex integration in space, modulo an error term of high spacial frequencies, the remaining error in \eqref{eq:heuristics_convex_spacetime} reduces to
\begin{equation}\label{eq:heuristics_high_tem_error}
R ( \widetilde{g}_{\kappa}  g_{\kappa}  -1), 
\end{equation}
which is a term of high temporal frequency and thus can be canceled by adding a temporal corrector $\theta_{o}$ such that to the leading order
\begin{equation}\label{eq:heuristics_high_tem_error_1}
\p_t \theta_o = - ( \widetilde{g}_{\kappa}  g_{\kappa}  -1)\D R .
\end{equation}

To see that $\theta_o$ is indeed small compared to $\theta$,  note that the error term \eqref{eq:heuristics_high_tem_error} has only low frequencies in space, and thus we have
$$\|\theta_o \|_{L^1_t L^p} \ll 1, 
$$
provided $\theta_o$ oscillates much faster in time than the old defect filed $R$,

\subsection{Other considerations}
Introducing temporal intermittency and oscillations comes at the cost of worse bounds in time for the perturbations $\theta$ and $w$. Of particular importance is the question of whether the iteration scheme can go through, i.e. the defect field $R$ can be made small in $L^1_{t,x}$. The most relevant part in the scheme is the term $R_{\Tem} $ solving the equation
\begin{equation}\label{eq:R_tem_error}
\D R_{\Tem}= \p_t \theta . 
\end{equation}
 It is clear that this term will impose certain constraints on the size of temporal frequencies. In the end, it is the potential $\mathbf{\Omega}_k$ in Theorem \ref{thm:peridoc_Phi_k_W_k} that saves the day: writing $\theta$ as a divergence of a potential allows us to gain one full derivative in space. We can also infer from \eqref{eq:R_tem_error} that the temporal frequency should be comparable to the spatial frequency.

Notice that \eqref{eq:temporal_intermittency_gain} does not require any intermittency in space but only intermittency in time. It turns out that as long as $\theta$ and $w$ are not homogeneous, i.e. just a little intermittent in space, \eqref{eq:temporal_intermittency_gain} can be achieved. The spacial intermittency is used to reconcile \eqref{eq:temporal_intermittency_gain} and \eqref{eq:R_tem_error} which is impossible when $\theta$ and $w$ are completely homogeneous.

This is quite surprising and very different than the idea used in \cite{MR3884855,MR4029736,1902.08521,2003.00539}, where more spacial intermittency for a solution yields improvements for the nonuniqueness range.

\subsection{Intermittent functions in time $\widetilde{g}_\kappa $ and $g_\kappa $}
In this subsection we shall define the intermittent oscillatory functions $\widetilde{g}_\kappa $ and $g_\kappa $. We take a profile function $g  \in C^\infty_c ([0,1])$ such that 
$$
\int_{[0,1]} g^2  \,dt =1.
$$

Let $\kappa \geq 1$ be a temporal concentration parameter that will be fixed in the next section. We introduce the temporal intermittency by adding a concentration utilizing  $\kappa$ as follows. 

Define  $g_{\kappa }: \RR \to \RR$ by
\begin{align} 
g_{\kappa } (t) =  g(\kappa  t).
\end{align}
Note that $\kappa \geq 1$ implies $\Supp g_{\kappa} \subset [0,1]$.
By a slight abuse of notation, $g_{\kappa }$ will also denote the $1$-periodic extension of $ g_{\kappa } $ by means of the Possion summation. 

Next, we define
\begin{align}
\widetilde{g}_{\kappa } = \kappa g_{\kappa }  ,
\end{align}
so that $g_{\kappa }, \widetilde{g}_{\kappa } : [0,1] \to \RR$ are both $1$-periodic. We will use $ \widetilde{g}_{\kappa }$ to oscillate the densities $\bp_k$, and $ g_{\kappa } $ to oscillate the vectors $ \bw_k$. Note the following important intermittency estimates
\begin{equation}
 \begin{aligned}\label{eq:def_g_theta_g_w}
\|  \widetilde{g}_{\kappa } \|_{L^r([0,1])} &\lesssim \kappa^{1 - \frac{1}{r} }, \\
\|  {g}_{\kappa } \|_{L^r([0,1])} &\lesssim \kappa^{  - \frac{1}{r} },
\end{aligned}   
\end{equation}
and the normalization identity
\begin{equation}\label{eq:g_theta_g_w_interaction}
\int_{[0,1]} \widetilde{g}_{\kappa }  {g}_{\kappa }  \,dt =1.
\end{equation}

Because of \eqref{eq:def_g_theta_g_w}, for any $k \in \NN$ we may choose $r>0$ such that
$$
 \int_0^1 \| w (t)\|_{W^{k,\infty}}^r \, dt  \ll 1,
$$
which confirms a note in Remark \ref{remark:intro} that the vector field $  u$  will concentrate on a small ``bad'' set in $[0,1] \times \TT^d$.

\subsection{Temporal correction function $h_\kappa $}

Finally, concerning the temporal corrector $\theta_o$ in \eqref{eq:heuristics_high_tem_error_1}, we define a periodic function $  h_{\kappa} :[0,1] \to \RR$ by
\begin{equation}
 h_{\kappa} (t):= \int_0^t (\widetilde{g}_{\kappa } {g}_{\kappa } -1) \, d\tau,
\end{equation}
so that
\begin{equation}\label{eq:def_h_mu_t}
\p_t  h_{\kappa} =   \widetilde{g}_{\kappa}   {g}_{\kappa}  -1.
\end{equation}

Note that by \eqref{eq:g_theta_g_w_interaction}, $ h_{\kappa}$  is well-defined and an approximation of a saw-tooth function, and we have the estimate
\begin{align}\label{eq:def_h_mu_t_L_1}
\|  h_{\kappa} \|_{L^\infty[0,1]} \leq 1,
\end{align}
which holds uniformly in $\kappa$.

In other words, $h_\kappa$ is not intermittent at all for any $\kappa>0$, and it will be used to design the temporal corrector $\theta_o$ in the next section.

\section{Proof of Proposition \ref{prop:main_prop}: defining perturbations and the defect field}\label{section:prop_proof1}
The main aim of this section is to define the perturbation density $\theta$ and velocity $w$, as well as to solve for the new defect filed $R_1$. This section is the core of the proof of Proposition \ref{prop:main_prop}.

Let us summarize the main steps of this section as follows.

\begin{enumerate}
    \item We first fix all the parameters for the building blocks $ (\bp_k , \bw_k)$ and $\widetilde{g}_\kappa, g_\kappa $ as explicit powers of $\lambda$, whose value we shall fix in the end.
    
    \item Next, we introduce a partition of the old defect field $R$ to ensure the smoothness of the perturbation.

    \item Then we define the perturbation $ (\theta , w)$ which, to the leading order, consists of linear combinations of the building blocks $(\bp_k , \bw_k) $ with suitable  coefficients that  oscillate intermittently in time using the functions  $\widetilde{g}_\kappa, g_\kappa$ defined in Section \ref{section:temporal_intermittency}.
    
    \item 
 Having defined the perturbations, we finally design the new defect field $R_1$ so that the new density $\rho+ \theta $ and the new vector field $u+w$ solve the continuity-defect equation with the new defect field $R_1$. 
\end{enumerate}

\subsection{Defining the parameters}
Given $p,q$ as in Proposition \ref{prop:main_prop}, there exists $\gamma>0$ such that
\begin{equation}\label{eq:def_gamma}
\min\{ 1 - \frac{1}{p}, \frac{1}{q} - \frac{1}{p'} \} > 4\gamma.
\end{equation}

Let $\l_0$ be the lower bound on $\mu$ given by Theorem \ref{thm:peridoc_Phi_k_W_k}. We fix the frequency parameters $\l, \mu,\kappa , \sigma >0$ as follows:
\begin{itemize}

\item The major frequency parameter
$$
\l \geq \l_0
$$
will be fixed at the end depending on the previous solution $(\rho,u,R) $ and the given parameters $\delta,\nu$ in Proposition \ref{prop:main_prop}.

\item Concentration parameters $\mu, \kappa$:
\begin{align*}
 \mu & = \kappa = \l.
\end{align*}

\item Oscillation parameter $\sigma   \in \NN $:

\begin{align*}
 \sigma   = \lfloor \l^{ \gamma}  \rfloor.
\end{align*}

\end{itemize}

Note that the space and time periodicity require $\sigma$ to be an integer. It is also worth noting that both  $\mu$ and $\sigma$ can be any positive powers of $\l $. In contrast, the temporal concentration $\kappa$ has to be almost a full spacial derivative.

Below is a direct consequence of the choice of parameters.
\begin{lemma}\label{lemma:bound_hierarchy}
There exists $r>1$ such that for any $\l \geq \l_0$, there holds
\begin{align*}
      \sigma   \mu^{\frac{d- 1}{p'}  -\frac{d-1}{q} } &\leq \l^{- 2\gamma},\\
        \mu^{\frac{d-1}{p} -  \frac{d-1}{r}  } &\leq
    \l^{-  \gamma} .
\end{align*}
\end{lemma}

\subsection{Defect field cutoff}
To ensure smoothness of the perturbation $(\theta , w)$, we shall avoid the region where $R$ is small. To this end, we introduce cutoffs based on each component of $R$. Denote by $R_k$ the $k$-th component of the old defect field $R $, i.e., we write
\begin{equation}
R(t,x)  = \sum_{1\leq k \leq d} R_k(t,x) \mathbf{e}_k.
\end{equation}

We specify the constant $r>0$ in Proposition \ref{prop:main_prop} as follows. Fix $ r>0$ sufficiently small such that
\begin{equation}\label{eq:def_constant_r}
\|R \|_{L^\infty([0,1] \times \TT^d)} \leq \frac{1}{4 r d}.
\end{equation}

Next, we define smooth cutoff functions $\chi_k \in C^\infty_c ( [0,1] \times \TT^d)$  such that
\begin{equation}\label{eq:def_chi_j}
0 \leq  \chi_k  \leq 1,\quad \chi_k(t,x) = 
\begin{cases}
0 & \text{if $|R_k| \leq \frac{\delta}{8d}$  or $t \not\in I_{r/2}$ }  \\
1 & \text{if $ |R_k| \geq  \frac{\delta}{4d}$ and  $t  \in I_{ r }$}.
\end{cases}
\end{equation}
where we recall the notation $I_r= [r, 1-r]  \subset [0,1]$. Note that by design each $\chi_k$ is also time-periodic.
Such cutoffs $\chi_k$ can easily be constructed by first cutting according to the size of $|R_k|$ and then multiplying by an additional cutoff in time. Also note that bounds for $\chi_k$ depend on $R$ and $\delta$.

Now let us cut off $R_k$ by introducing \begin{equation}
\widetilde{R}_k = \chi_k R_k.
\end{equation}
In what follows we often use the crude bounds
\begin{align}
|\nabla_{t,x}^n \widetilde{R}_k  | \lesssim_{R ,n,\delta} 1.
\end{align}

\subsection{Density and velocity perturbation $(\theta , w)$}\label{subsec:theta_w_definition}
The idea in the construction of the perturbation $(\theta , w)$ is the use of $d$-pairs of disjoint $(\bp_k ,\bw_k) $ to cancel each component $R_k$ on average in time by means of the intermittent oscillating factors in Section \ref{section:temporal_intermittency}.

In summary, the perturbations $(\theta , w)$ are defined by
\begin{align*}
    \theta &= \theta_p +\theta_c + \theta_o \\
    w &= w_p + w_c
\end{align*}
where $\theta_p $ and $ w_p$ are principle parts of the perturbation, $\theta_c $ and $ w_c$ are correctors for the zero mean of $\theta$ and zero divergence of $w$ respectively, and $\theta_o $ is a zero mean oscillator that we use to balance a high temporal frequency error later.

We first define the principle part of the perturbations. Let 
\begin{align}\label{eq:def_theta_p_w_p_long}
\theta_p(t,x)  &:=  \nu^{-1}  \widetilde{g}_{\kappa} (\sigma t)\sum_{1\leq k \leq d}   \frac{\| \widetilde{R}_k (t) \|_{1}^{1 - \frac{1}{p}}  }{\| \widetilde{R}_k \|_{L^1_{t,x}}^{1 - \frac{1}{p}} } \sgn(- R_k)  \chi_k |R_k|^{\frac{1}{p}} \bp_k(\sigma x), \\
w_p (t,x)  &:=  \nu   g_{\kappa} (\sigma  t)  \sum_{1\leq k \leq d} \frac{\|  \widetilde{R}_k (t) \|_{1}^{   - \frac{1}{p'}} }{\|  \widetilde{R}_k  \|_{L_{t,x}^1}^{   - \frac{1}{p'}}  }      \chi_k |R_k|^{\frac{1}{p'}} \bw_k (\sigma x) .
\end{align}

The smoothness of $\theta_p $ and $w_p$ will be proved in Lemma \ref{lemma:a_k}. We take a moment to analyze the role of each part involved in the definition. 
\begin{itemize}
\item The factors $\|\widetilde{R}_k(t) \|_{1}^{1 - \frac{1}{p}} $ and $\|\widetilde{R}_k (t) \|_{1}^{  - \frac{1}{p'}} $ are for the normalization to kill the old defect field $R$ via the high-high to low interactions in space .
\item The cutoffs $\chi_k$ is to ensure smoothness by avoiding the regime where $R_k$ is small. Note that if $\| \widetilde{R}_k \|_{L^1_{t,x}} = 0$, then $ \| R \|_{L^\infty_{t,x}} \leq  \delta $ and there is nothing to prove.
\item The building blocks $\bp_k(\sigma x)$ and $\bw_k (\sigma x) $ are used to perform the convex integration in space, similar to the previous works.

\item Finally, $ \widetilde{g}_{\kappa} (\sigma t)$ and $ g_\omega (\sigma t)    $ are  factors that encode  the temporal intermittency and oscillations. We will then perform a ``convex integration in time'' to kill the error of high temporal frequency.

\end{itemize}

For brevity, let us introduce shorthand notations
\begin{align}\label{eq:def_theta_p_w_p_short}
\theta_p(t,x)   &  =\nu^{-1}  \widetilde{g}_{\kappa} (\sigma t)\sum_{1\leq k \leq d}  A_k(t,x)  \bp_k(\sigma x),\\
w_p (t,x)  & =  \nu   g_{\kappa} (\sigma  t)   \sum_{1\leq k \leq d}B_k(t,x) \bw_k (\sigma x), 
\end{align}
where
\begin{align}\label{eq:def_A_k_B_k}
A_k(t,x) &= \frac{\| \widetilde{R}_k (t) \|_{1}^{1 - \frac{1}{p}}  }{\| \widetilde{R}_k \|_{L^1_{t,x}}^{1 - \frac{1}{p}} }  \chi_k \sgn(-R_k)  |R_k|^{\frac{1}{p}}, \\
B_k(t,x) &=  \frac{\|  \widetilde{R}_k (t) \|_{1}^{   - \frac{1}{p'}} }{\| \widetilde{R}_k  \|_{L_{t,x}^1}^{   - \frac{1}{p'}}  }   \chi_k   |R_k|^{\frac{1}{p'}}.
\end{align}
Note the important identity that motivates our choice of $ A_k  $ and  $B_k $:
\begin{equation}\label{eq:def_A_k_B_k_product}
 A_k B_k =   -\chi_k^2 R_k \quad \text{for all $k=1,\dots,d$}.
\end{equation}

In view of the zero-mean  requirement  for $\theta$ and the divergence-free condition for $w$, we introduce correctors
\begin{align}\label{eq:def_theta_c_w_c}
\theta_c(t)  & :=    \fint_{\TT^d }\theta_p(t,x) \,dx,  \\
w_c (t,x)  &:=  -\nu   g_{\kappa} (\sigma  t)   \sum_{1\leq k \leq d} \mathcal{B}  ( \nabla B_k   , \bw_k (\sigma \cdot)  ) . 
\end{align}
where $ \mathcal{B} $ is the bilinear antidivergence operator in Lemma \ref{lemma:cheapbound_B}.

Since $ \nabla B_k   \cdot  \bw_k   =\D(B_k \bw_k) $ has zero mean, by a direct computation 
$$
\D w_c  =  -\nu   g_{\kappa} (\sigma  t)   \sum_{1\leq k \leq d} \D \mathcal{B}  ( \nabla B_k   , \bw_k (\sigma \cdot)  )   =- \D w_p.
$$
Thanks to Theorem \ref{thm:peridoc_Phi_k_W_k} and Lemma \ref{lemma:cheapbound_B}, these two correctors are small compared to the principle parts $\theta_p$ and $w_p$.

Finally, we take advantage of the temporal oscillations and define a temporal oscillator
\begin{equation}\label{eq:def_theta_o_1}
\theta_o (t,x):= \sigma^{-1} h(\sigma t) \sum_{1\leq k \leq d} \Big( \fint_{\TT^d } \bp_k \bw_{k } \, dx   \Big) \cdot  \nabla (\chi_k^2 R_k),
\end{equation}
which, thanks to Lemma \ref{thm:peridoc_Phi_k_W_k}, is equivalent to
\begin{equation}\label{eq:def_theta_o_2}
\theta_o   = \sigma^{-1} h(\sigma t) \D \sum_{1\leq k \leq d}   \chi_k^2 R_k   \ek    .
\end{equation}
The role of the temporal oscillator $\theta_o$ is to balance the high temporal frequency error by its time derivative in the convex integration scheme, which will be done in Lemma \ref{lemma:final_decomposition_R_osc}. Note that by definition $\theta_o $ has zero spatial mean at each time.

\subsection{The new defect field $R_1$ }

Our next goal is to define a suitable defect field $R_1$ such that the new density $\rho_1$ and vector field $u_1$,
$$
\rho_1 : = \rho + \theta, \qquad u_1 : = u+ w,
$$
solve the continuity-defect equation
\begin{equation}\label{eq:new_equation_R_1}
\p_t \rho_1 + u_1 \cdot \nabla \rho_1  = \D R_1.
\end{equation}
To do so, we will solve the divergence equations
\begin{align*}
\D R_{\Osc}  & =   \D (\theta_p w_p + R ) + \p_t \theta_o,\\
\D R_{\Tem} &=   \p_t (\theta_p   +  \theta_c ),\\
\D R_{\Lin}  & =   \D (\theta  u + \rho w),  \\
\D R_{\Cor}  & =     \D  \big( \theta  w_c \big) + \D \big( (\theta_o +  \theta_c ) w_p  \big),
\end{align*}
so that $R_1 =  R_{\Osc} +  R_{\Tem} + R_{\Lin} + R_{\Cor}$.

The choices for $ R_{\Lin}$ and $ R_{\Cor} $ are relatively straightforward. 
\begin{definition}\label{def:new_R_1}
The new defect field $R_1$ is defined by
$$
R_1= R_{\Osc} + R_{\Lin}  + R_{\Cor} + R_{\Tem},  
$$
where $R_{\Lin}$   and  $R_{\Cor}$ are defined by
\begin{align*}
R_{\Lin}  & :=   \theta u + \rho w, \\
R_{\Cor}  & :=      \theta  w_c + (\theta_o +  \theta_c ) w,
\end{align*}
while $ R_{\Tem}$ and $R_{\Osc} $ are defined in Lemma~\ref{lemma:def_R_tem} and Lemma~\ref{lemma:final_decomposition_R_osc} respectively .
\end{definition}

Next, we specify the choice for $ R_{\Tem}$, which utilizes the bilinear antidivergence operator $\mathcal{B}$.

\begin{lemma}\label{lemma:def_R_tem}
Let
$$
R_{\Tem}: = \nu^{-1} \p_t \Big(   \widetilde{g}_{\kappa} (\sigma t)\sum_{1\leq k \leq d}  \mathcal{B}( A_k,\bp_k(\sigma \cdot)) \Big).
$$
Then 
$$
 \p_t (\theta_p+ \theta_c)  =\D R_{\Tem} .
$$

\end{lemma}
\begin{proof}

Note that
$$
\theta_p + \theta_c = \nu^{-1}      \widetilde{g}_{\kappa} (\sigma t)\sum_{1\leq k \leq d} \big(   A_k \bp_k(\sigma \cdot)  -\fint_{\TT^d}  A_k \bp_k(\sigma  \cdot)    \big) .
$$
Then the conclusion follows immediately from the definition of $ \mathcal{B}$ and the fact that $\bp_k$ has zero mean, cf. \eqref{eq:bilinear_B_identity}.
\end{proof}

\subsection{Convex integration in space-time: designing $ R_{osc} $ }
This subsection is the core of our convex integration scheme. The main goal is to design a suitable oscillation part $R_{\Osc}$ of the defect field so that
$$
\D R_{\Osc} = \D(\theta_p w_p +R) +\p_t \theta_o.
$$

To this end, we first isolate terms in the nonlinearity $\D(\theta_p w_p +R)$ according to their roles, and then use the temporal corrector $\p_t \theta_o$ to balance the part with high temporal frequencies in $\D(\theta_p w_p +R)$.

\begin{lemma}[Space-time oscillations]\label{lemma:convex_integration_space_time}
The following identity holds
\begin{align}
 \D (\theta_p w_p + R ) = \D \big( R_{\text{osc,x}} +R_{\text{hi,t}}  + R_{\Rem} \big),
\end{align}
where $R_{\text{osc,x}} $ is the oscillation error in space
$$
R_{\text{osc,x}} =  \widetilde{g}_{\kappa}  (\sigma t)  g_{\kappa} (\sigma  t)  \sum_{1\leq k \leq d} \mathcal{B} \Big(  \nabla(A_k B_k) , \big(  \bp_k \bw_k (\sigma x) -\fint_{\TT^d}  \bp_k \bw_k \, dx \big)  \Big)  ,
$$
$R_{\text{hi,t}} $ is the error of high frequency in time
$$
R_{\text{hi,t}}   = \Big(  \widetilde{g}_{\kappa}  (\sigma t)  g_{\kappa} (\sigma  t) - \fint_{[0,1]}  \widetilde{g}_{\kappa}    g_{\kappa}  \Big)\sum_{1\leq k \leq d}  A_k B_k \fint_{\TT^d}  \bp_k \bw_k \, dx   ,
$$
and $R_{\Rem}$ is the remainder error
$$
R_{\Rem}= \sum_{1\leq k \leq d}  (1-\chi_k^2 ) R_k \ek.
$$
\end{lemma}
 
\begin{proof}

By the definition of $\theta_p$, we have   
\begin{align}
\theta_p w_p =   \widetilde{g}_{\kappa}  (\sigma t)  g_{\kappa} (\sigma  t) \sum_{1\leq k \leq d}  A_k B_k  \bp_k \bw_k (\sigma \cdot )    .
\end{align}

Taking divergence, we have
\begin{align}\label{eq:convex_integration_space_time_1}
\D (\theta_p w_p +R) &=   \widetilde{g}_{\kappa}  (\sigma t)  g_{\kappa} (\sigma  t) \sum_{1\leq k \leq d} \D\big(  A_k B_k  \bp_k \bw_k (\sigma \cdot )  \big)   + \D R .
\end{align}
By the last point in Theorem \ref{thm:peridoc_Phi_k_W_k},
\begin{align*}
 \D\big(  A_k B_k  \bp_k \bw_k (\sigma \cdot )  \big) & =  \D\Big(  A_k B_k  \big( \bp_k \bw_k (\sigma \cdot ) -\fint_{\TT^d} \bp_k \bw_k  \big) \Big) + \D (A_k B_k \ek) \\
 & =     A_k B_k  \D \big( \bp_k \bw_k (\sigma \cdot ) \big)  +  \nabla( A_k B_k ) \cdot \big( \bp_k \bw_k (\sigma \cdot ) -\fint_{\TT^d} \bp_k \bw_k  \big)  + \D (A_k B_k \ek ) ,
\end{align*}
where the first two terms combined together have zero mean.
For this reason, thanks to the definition of $\mathcal{B}$, we may write
\begin{align}
 \D\big(  A_k B_k  \bp_k \bw_k (\sigma \cdot )  \big)  
 & =    \D \mathcal{B} \Big( A_k B_k , \D \big( \bp_k \bw_k (\sigma \cdot ) \big) \Big) \nonumber \\
 &\qquad + \D \mathcal{B} \Big( \nabla( A_k B_k ) , \big( \bp_k \bw_k (\sigma \cdot ) -\fint_{\TT^d} \bp_k \bw_k  \big)  \Big) + \D (A_k B_k \ek ) \label{eq:convex_integration_space_time_2}.
\end{align}
Now it follows from \eqref{eq:convex_integration_space_time_1} and \eqref{eq:convex_integration_space_time_2} that
\begin{align}\label{eq:convex_integration_space_time_3}
\D (\theta_p w_p +R) &=   \widetilde{g}_{\kappa}  (\sigma t)  g_{\kappa} (\sigma  t) \sum_{1\leq k \leq d}  \D (A_k B_k \ek ) + \D R + \D R_{\text{osc,x}}      .
\end{align}

To see that $
\D ( \theta_p w_p + R )=  \D (R_{\text{osc,x}} + R_{\text{hi,t}}      + R_{\Rem} )
$, by an examination of \eqref{eq:convex_integration_space_time_3}, we need to show that
$$
R_{\Rem} = R + \fint_{[0,1]}  \widetilde{g}_{\kappa}    g_{\kappa} \sum_{1\leq k \leq d}  A_k B_k \fint_{\TT^d}  \bp_k \bw_k \, dx .
$$
Indeed, using \eqref{eq:g_theta_g_w_interaction}, Theorem \ref{thm:peridoc_Phi_k_W_k}, and \eqref{eq:def_A_k_B_k_product} we obtain that
\begin{align}
\fint_{[0,1]}  \widetilde{g}_{\kappa}    g_{\kappa} \sum_{1\leq k \leq d}  A_k B_k \fint_{\TT^d}  \bp_k \bw_k \, dx  = \sum_{1\leq k \leq d}  A_k B_k \ek, 
\end{align}
which implies that
$$
R+ \sum_{1\leq k \leq d}  A_k B_k \ek  = R - \sum_{1\leq k \leq d}  \chi_k^2  R_k \ek = \sum_{1\leq k \leq d}  (1-\chi_k^2 ) R_k \ek =R_{\Rem}.
$$

\end{proof}

Due to the designed temporal corrector $\theta_o$, the error of high frequency in time $R_{\text{hi,t}}$ is canceled to the leading order by $\p_t \theta_o$. We complete the design of the oscillation error $R_{\Osc}$ in the following lemma.

\begin{lemma}\label{lemma:final_decomposition_R_osc}
Let
\begin{align}
  R_{\Osc}   : =   R_{\text{osc,x}}  + R_{\text{osc,t}}  +  R_{\Rem}  ,
\end{align}
where $R_{\text{osc,x}}$  and $R_{\Rem} $ are as in Lemma \ref{lemma:convex_integration_space_time}, and $R_{\text{osc,t}}   $ is the oscillation error in time
$$
R_{\text{osc,t}} =  \sigma^{-1} h(\sigma t) \mathcal{R}  \sum_{1\leq k \leq d} \ek \cdot  \p_t \nabla (\chi_k^2 R_k).
$$

Then the oscillation error $R_{\Osc}   $ satisfies the identity
$$
\D R_{\Osc} =   \D (\theta_p w_p + R ) + \p_t \theta_o  .
$$
\end{lemma}
\begin{proof}
By the previous lemma, we only need to verify that
$$
\p_t \theta_o  + \D R_{\text{hi,t}} =  \D R_{\text{osc,t}}.
$$
By the definition of $\theta_o$ \eqref{eq:def_theta_o_2}, we have
$$
\p_t \theta_o =   \p_t h (\sigma t) \sum_{1\leq k \leq d} \ek \cdot  \nabla (\chi_k^2 R_k) + \sigma^{-1} h(\sigma t) \sum_{1\leq k \leq d} \ek \cdot  \p_t\nabla  (\chi_k^2 R_k).
$$

It follows from \eqref{eq:def_h_mu_t} and \eqref{eq:def_A_k_B_k_product} that
$$
 \p_t h (\sigma t) \sum_{1\leq k \leq d} \ek \cdot  \nabla (\chi_k^2 R_k) = \Big(  \widetilde{g}_{\kappa}  (\sigma t)  g_{\kappa} (\sigma  t) - \fint_{[0,1]}  \widetilde{g}_{\kappa}    g_{\kappa}  \Big)\sum_{1\leq k \leq d} \ek \cdot \nabla ( -A_k B_k  ) =-\D R_{\text{hi,t}}, 
$$
which implies that
$$
\p_t \theta_o + \D R_{\text{hi,t}} =  \D R_{\text{osc,t}}.
$$

\end{proof}

\subsection{Verification of $(u_1, \rho_1, R_1)$ as a solution of the continuity-defect equation}
We conclude this section by showing that the new solution $(u_1, \rho_1, R_1)$ is indeed a solution to the continuity-defect equation.
\begin{lemma}\label{lemma:conclusion_new_rho_u_R}
The density $\rho_1 =\rho+\theta$, vector field $u_1 = u+ w$, and defect field $R_1 =R_{\Lin} + R_{\Tem} + R_{\Cor} + R_{\Osc}$ solve the equation
$$
\p_t \rho_1 + u_1 \cdot \nabla \rho_1 = \D R_1.
$$

Moreover, the temporal support of the perturbation $(\theta,w)$ satisfies
$$
\Supp_t \theta \cup \Supp_t w \subset \Supp_t R.
$$
\end{lemma}
\begin{proof}
We compute 
\begin{align*}
 \p_t \rho_1 + u_1 \cdot \nabla \rho_1 & = (\p_t \rho + u \cdot \nabla \rho ) + ( \p_t \theta + \D( \theta u) + \D(\theta w )  + \D(\rho w  ) )\\
 & = \D R + \p_t \theta + \D( \theta u) + \D(\theta w )  + \D(\rho w  ).
\end{align*}
Now the first claim follows from Definition \ref{def:new_R_1}, Lemma \ref{lemma:def_R_tem}, and Lemma \ref{lemma:final_decomposition_R_osc}.

The second claim follows from the definitions \eqref{eq:def_theta_p_w_p_long}, \eqref{eq:def_theta_c_w_c}, and \eqref{eq:def_theta_o_1} since the coefficients vanish whenever $R(t)$ vanishes.
\end{proof}

To complete the proof of Proposition \ref{prop:main_prop}, it remains to verify the estimates for the perturbation $(\theta, w)$ and the new defect field $R_1$. We do this in Section \ref{section:prop_proof2} for the perturbation and  in Section \ref{section:prop_proof3} for the new defect field respectively.

\section{Proof of the proposition \ref{prop:main_prop}: estimates on the perturbation}\label{section:prop_proof2}

In this section, we will derive estimates for the perturbation $(\theta, w)$. The main tools have been listed in Section \ref{section:pre}. The main idea is to take the frequency parameter $\l$ sufficiently large depending on the previous solution $(\rho,u,R)$ so that the error terms are negligible. It is also worth noting that all implicit constants will not depend on $(\rho,u,R)$ unless otherwise indicated.

We start with the smoothness and time periodicity of the coefficients $A_k,B_k$, which are necessary conditions for Lemma \ref{lemma:improved_Holder} and \ref{lemma:mean_values_oscillation}.
\begin{lemma}[Smoothness of $A_k,B_k$]\label{lemma:a_k}
The coefficients $A_k,B_k \in C^\infty([0,1] \times \TT^d) $ are time-periodic on $[0,1]$, and the map 
\begin{equation}
t \mapsto \|  \widetilde{R}_k (t)\|_{L^1(\TT^d ) }   
\end{equation}
is smooth on $[0,1]$.  In particular, all the perturbations $\theta_p, \theta_c, \theta_o$ and $ w_p , w_c$ are smooth and time-periodic.

Moreover, the following estimates hold uniformly in time
\begin{align*}
\|A_k (t)   \|_{L^p(\TT^d ) }  & \leq \|  \widetilde{R}_k \|_{L^1_{t,x}}^{-1 + \frac{1}{p}} \|  \widetilde{R}_k (t) \|_{L^1(\TT^d ) }, \\
\|B_k (t)   \|_{L^{p'}(\TT^d ) } & \leq \|  \widetilde{R}_k \|_{L^1_{t,x}}^{  \frac{1}{p'}}    .
\end{align*}
\end{lemma}
\begin{proof}
Denote by $ R_k^+ = \max\{ R_k, 0 \} $ and $R_k^- = \min\{ R_k, 0  \} $. Due to the cutoff $\chi_k$, the functions
\begin{equation}\label{eq:smoothness_proof}
  \chi_k    R_k^{\pm}   
\end{equation}
are smooth on $ [ 0,1] \times \TT^d$.  Thus the map
$$
t \mapsto \| \widetilde{R}_k (t) \|_{L^1(\TT^d)} = \int \chi_k R_k^+  - \chi_k R_k^- \,dx 
$$
is smooth on $[0,1]$. 

Next, let us show that the coefficients $A_k$ and $B_k$ are smooth on $[0,1] \times \TT^d$. Indeed, due to the smoothness of $\|\widetilde{R}_k(t)\|_1$, the coefficients
$A_k, B_k$ are automatically smooth at all points where
$\|\widetilde{R}_k(t)\|_1>0$.
On the other hand, for any 
point $(t,x)$, where $\|R_k(t)\|_1=0$, there is a neighborhood of of that point where $\chi_k\equiv 0$. Hence,  $A_k \equiv B_k \equiv 0$ in that neighborhood. Therefore, $A_k,B_k \in C^\infty([0,1] \times \TT^d) $. Their time-periodicity follows simply from the definitions.

Finally, we show the pointwise $L^p$ and $L^{p'}$ estimates for $A_k$ and $B_k$. For $A_k$ we have
\begin{align*}
\| A_k (t)\|_{L^p(\TT^d)} & \leq  \frac{\| \widetilde{R}_k  (t) \|_{1}^{1 - \frac{1}{p}}  }{\| \widetilde{R}_k \|_{L^1_{t,x}}^{1 - \frac{1}{p}} } \|  \chi_k   |R_k |^{\frac{1}{p}}\|_{L^p(\TT^d)}\\
& \leq  \frac{\| \widetilde{R}_k  (t) \|_{1}^{1 - \frac{1}{p}}  }{\| \widetilde{R}_k \|_{L^1_{t,x}}^{1 - \frac{1}{p}} } \|  \chi_k^p    R_k   \|_{L^1(\TT^d)}^{\frac{1}{p}}\\
&\leq \frac{ \| \widetilde{R}_k (t) \|_{1} }{\| \widetilde{R}_k \|_{L^1_{t,x}}^{1 - \frac{1}{p}} } ,
\end{align*}
where we have used the fact that  $p\in (1,\infty)$.
The estimate for $B_k$ can be deduced in the same way:
\begin{align*}
\| B_k (t)\|_{L^{p'}(\TT^d)}   & \leq  \frac{\| \widetilde{R}_k  (t) \|_{1}^{  - \frac{1}{p'}}  }{\| \widetilde{R}_k \|_{L^1_{t,x}}^{  - \frac{1}{p'}} } \|  \chi_k   |R_k |^{\frac{1}{p'}} \|_{L^{p'}(\TT^d)} \\
 & \leq  \frac{\| \widetilde{R}_k  (t) \|_{1}^{  - \frac{1}{p'}}  }{\| \widetilde{R}_k \|_{L^1_{t,x}}^{  - \frac{1}{p'}} } \|  \chi_k^{p'}    R_k   \|_{L^{1}(\TT^d)}^{\frac{1}{p'}} \\
 & \leq   \| \widetilde{R}_k \|_{L^1_{t,x}}^{    \frac{1}{p'}}    .
\end{align*}
\end{proof}

\subsection{Estimates for the density $\theta$}

Here and in what follows, $C_R$ represents a positive constant that depends on the old defect field $R$ that may change from line to line.

\begin{lemma}[Estimate on $\theta_p$]\label{lemma:theta_p_estimate}
There holds
\begin{equation*}
\|\theta_p  \|_{L^1_t L^p} \lesssim \nu  \| R \|_{L^1_{t,x}}^\frac{1}{p}  + C_R  \sigma^{-\frac{1}{p} } .
\end{equation*}

In particular, for $\l$ sufficiently large,
\begin{equation*}
\|\theta_p  \|_{L^1_t L^p} \lesssim \nu  \| R \|_{L^1_{t,x}}^\frac{1}{p}   .
\end{equation*}
\end{lemma}
\begin{proof}
We first take $L^p$ norm in space, using the shorthand notation
\begin{align}\label{eq:proof_theta_p_1}
\| \theta_p(t ) \|_{  L^p(\TT^d)}  & \leq   \nu^{-1} \big|  \widetilde{g}_{\kappa} (\sigma t)\big| \sum_{1\leq k \leq d}       \big \|A_k(t) \bp_k(\sigma \cdot )   \big \|_{  L^p(\TT^d)}  .
\end{align}
Since for each fixed $t$, $A_k(t,x)  $ is smooth on $\TT^d$, by Lemma \ref{lemma:improved_Holder}, we have
\begin{equation}\label{eq:proof_theta_p_2}
 \Big \|A_k(t) \bp_k(\sigma \cdot )   \Big \|_{  L^p(\TT^d)}     \leq        \big \| A_k(t )  \big \|_{  L^p(\TT^d)}      \| \bp_k(\sigma \cdot )     \|_{   p }  + C_R  \sigma^{-\frac{1}{p} }      \| \bp_k     \|_{   p }  .        
\end{equation}
 By Proposition \ref{prop:bounds_periodic_phi_w} and Lemma \ref{lemma:a_k}, combining \eqref{eq:proof_theta_p_1} and \eqref{eq:proof_theta_p_2} we obtain
\begin{align*}
\| \theta_p(t ) \|_{  L^p(\TT^d)}    \lesssim  \nu^{-1}  \widetilde{g}_{\kappa} (\sigma t)  \sum_{1\leq k \leq d}     \|  \widetilde{R}_k \|_{L^1_{t,x}}^{-1 + \frac{1}{p}}    \|  \widetilde{R}_k  (t) \|_{   1 }  + C_R \sigma^{-\frac{1}{p} }   .
\end{align*}
Now we take $L^1$ in time to obtain
\begin{align}\label{eq:proof_theta_p_3}
\| \theta_p  \|_{ L^1_t L^p }    &\lesssim  \nu^{-1} \sum_{1\leq k \leq d}  \|  \widetilde{R}_k   \|_{L^1_{t,x}}^{-1 + \frac{1}{p}}  \int_{[0,1]}  \big|  \widetilde{g}_{\kappa} (\sigma t)\big|  \| \widetilde{R}_k (t) \|_{   1 } \, dt + C_R \sigma^{-\frac{1}{p} }   .
\end{align}
Given the smoothness of $t \to \| \widetilde{R}_k  (t) \|_{   1 }$ proven in Lemma \ref{lemma:a_k}, applying Lemma \ref{lemma:improved_Holder} once again (in time) gives
\begin{align}\label{eq:proof_theta_p_4}
\int_{[0,1]}  \big|  \widetilde{g}_{\kappa} (\sigma t)\big|  \|  \widetilde{R}_k  (t) \|_{   1 } \, dt \lesssim    \|  \widetilde{R}_k   \|_{L^1_{t,x}}  \big\|  \widetilde{g}_{\kappa} \big  \|_{ L^1([0,1]) }        + C_R \sigma^{-\frac{1}{p} }   .
\end{align}
Then it follows from  \eqref{eq:proof_theta_p_3} and \eqref{eq:proof_theta_p_4} that
\begin{align*}
\| \theta_p \|_{ L^1_t L^p }   & \lesssim \nu^{-1} \sum_{1\leq k \leq d}  \|R_k  \|_{L^1_{t,x}}^{ \frac{1}{p}}  \big\|  \widetilde{g}_{\kappa} \big  \|_{ L^1([0,1]) }        + C_R \sigma^{-\frac{1}{p} }  \\
& \lesssim \nu^{-1}   \|R   \|_{L^1_{t,x}}^{ \frac{1}{p}}          + C_R \sigma^{-\frac{1}{p} }  ,  
\end{align*}
where we have also used \eqref{eq:def_g_theta_g_w}.

Once we take $\l$ sufficiently large such that the error term $$  
C_R \sigma^{-1}    \leq \nu^{-1}    \big \|  R  \|_{  L^1_{t,x} }^{ \frac{1}{p}} ,
$$
the desired bound follows 
\begin{align*}
\| \theta_p \|_{ L^1_t L^p(\TT^d)}   \lesssim \nu^{-1}    \big \|  R  \|_{  L^1_{t,x} }^{ \frac{1}{p}},    
\end{align*} 
with an implicit constant independent of $\l, R$ and $\nu$.

\end{proof}

\begin{lemma}[Estimate on $\theta_c$]\label{lemma:theta_c_estimate}
There holds
\begin{equation*}
\|\theta_c  \|_{L^1_t L^p} \leq C_R  \nu^{- 1} \sigma^{-1} .
\end{equation*}
In particular, for $\l$ sufficiently large
$$
\|\theta_c  \|_{L^1_t L^p} \leq \nu^{-1}    \big \|  R  \|_{  L^1_{t,x} }^{ \frac{1}{p}} .   
$$

\end{lemma}
\begin{proof}
Since 
$$
\theta_c =  \nu^{-1}  \widetilde{g}_{\kappa} (\sigma t)\sum_{1\leq k \leq d} \fint_{\TT^d } A_k(t,x)  \bp_k(\sigma x) \,dx,
$$
this follows directly from Lemma \ref{lemma:mean_values_oscillation}.
\end{proof}

\begin{lemma}[Estimate on $\theta_o$]\label{lemma:theta_o_estimate}
There holds
\begin{equation*}
\|\theta_o  \|_{L^\infty_{t,x}} \leq C_R  \sigma^{-1} .
\end{equation*}
In particular, for $\l$ sufficiently large
$$
\|\theta_o  \|_{L^1_t L^p} \leq \nu^{-1}    \big \|  R  \|_{  L^1_{t,x} }^{ \frac{1}{p}} .   
$$
\end{lemma}
\begin{proof}
By \eqref{eq:def_theta_o_2}, H\"older's inequality and \eqref{eq:def_h_mu_t_L_1} we have
\begin{align*}
\| \theta_o \|_{L^\infty_{t,x}} & \leq \sigma^{-1} \big\|  h(\sigma \cdot ) \big\|_{L^\infty([0,1])} \sum_{1\leq k \leq d} \big\| \ek \cdot  \nabla R_k  \big\|_{L^\infty_{t,x}} \\
&\leq C_R \sigma^{-1} .   
\end{align*}

\end{proof}

\subsection{Estimates for the vector field $w$}
The vector field $w$ can also be estimated using the tools in Section \ref{section:pre}.

\begin{lemma}[Estimate on $w_p$]\label{lemma:w_p_estimate}
There holds
\begin{align*}
\| w_p  \|_{L^\infty_t L^{p'}} &\lesssim \nu \| R \|_{L^1_{t,x}}^\frac{1}{p'}  +C_R \sigma^{-\frac{1}{p'} }  \\
\| w_p  \|_{L^1_t W^{1,q}} &\leq  \nu C_R  \l^{- \gamma} .
\end{align*}

In particular, for $\l$ sufficiently large,
\begin{align*}
\| w_p  \|_{L^\infty_t L^{p'}} &\lesssim \nu \| R \|_{L^1_{t,x}}^\frac{1}{p'}    \\
\| w_p  \|_{L^1_t W^{1,q}} &\leq \delta/2 .
\end{align*}
\end{lemma}
\begin{proof}
We first prove the $L^\infty_t L^{p'}$ estimate, and then the Sobolev estimate $L^1_t   W^{1,q}$.

\begin{enumerate}[leftmargin=*]
    
\item \underline{\textit{$L^\infty_t L^{p'}$ estimates:}}

Taking the $L^{p'}$ norm in space yields
\begin{align}\label{eq:proof_w_p_1}
\| w_p (t ) \|_{p'}   &\leq   \nu  \big|   g_{\kappa} (\sigma  t) \big|   \sum_{1\leq k \leq d}     \big \|B_k(t )  \bw_k (\sigma \cdot )   \big\|_{p'}.
\end{align}
Since $x \mapsto  B_k(t,x)   $ is smooth on $\TT^d$ for all fixed $t \in [0,T]$, by Lemma \ref{lemma:improved_Holder}, we have
\begin{align}\label{eq:proof_w_p_2}
\Big \|B_k(t )  \bw_k (\sigma \cdot )  \Big\|_{p'} \leq        \big \|B_k(t ) \big\|_{p'} \big\|   \bw_k   \big\|_{p'}  + \sigma^{-\frac{1}{p'} } C_R   \| \bw_k     \|_{p'}.
\end{align}
Then from \eqref{eq:proof_w_p_1}, \eqref{eq:proof_w_p_2}, Lemma \ref{lemma:a_k},  and the fact that $ \big\|   \bw_k    \big\|_{p'}  \sim 1$, it follows that
\begin{align}
\| w_p (t ) \|_{p'}   &\lesssim   \nu  \big|   g_{\kappa} (\sigma  t) \big|  \sum_{1\leq k \leq d}   \|  R_k  \|_{L^1_{t,x}}^{    \frac{1}{p'}}          + \sigma^{-\frac{1}{p'} } C_R      .
\end{align}
Now we simply take $L^\infty$ in time to obtain
\begin{align*}
\| w_p  \|_{ L^\infty_t L^{p'}(\TT^d)}    
&\lesssim  \nu   \|R  \|_{L^1_{t,x}}^{   \frac{1}{p'}}   + C_R \sigma^{-\frac{1}{p'} } ,
\end{align*}
where we have used \eqref{eq:def_g_theta_g_w}.

Once we take $\l$ sufficiently large such that the error term $$  
C_R \sigma^{-1}  \leq \nu    \big \|  R  \|_{  L^1_{t,x} }^{   \frac{1}{p'}},
$$
the desired bound follows 
\begin{align*}
\| w_p \|_{ L^\infty_t L^p(\TT^d)}   \lesssim \nu   \big \|  R  \|_{  L^1_{t,x} }^{   \frac{1}{p'}}    .
\end{align*} 

\item \underline{\textit{Sobolev  estimate $L^1_t   W^{1,q}$:}}

 Taking Sobolev norm $W^{1,q} $ in space we have
\begin{align}\label{eq:w_p_proof_2}
\|  w_p (t ) \|_{W^{1,q}(\TT^d)}  \leq \nu  \big|  g_{\kappa} (\sigma  t)  \big| \sum_{1\leq k \leq d}   \Big\|  B_k(t) \bw_k (\sigma \cdot) \Big\|_{W^{1,q}(\TT^d)}.
\end{align}
Direct computation using H\"older's inequality gives
\begin{align*}
\Big\|  B_k(t) \bw_k (\sigma \cdot) \Big\|_{W^{1,q}(\TT^d)} \leq  C_R \Big(   \big\|    \bw_k (\sigma \cdot) \big\|_{L^{q}(\TT^d)}  +    \sigma \big\|    \nabla \bw_k (\sigma \cdot) \big\|_{L^{q}(\TT^d)} \Big) .  
\end{align*}
From this, by Proposition \ref{prop:bounds_periodic_phi_w}, we get
\begin{align}\label{eq:w_p_proof_3}
\Big\|  B_k(t) \bw_k (\sigma \cdot) \Big\|_{W^{1,q}(\TT^d)} \lesssim  C_R \sigma \mu^{1+\frac{d-1}{p'}  - \frac{d-1}{q}}.  
\end{align}

Thus from \eqref{eq:w_p_proof_2} and \eqref{eq:w_p_proof_3} we get
\begin{align}\label{eq:w_p_proof_4}
\|  w_p (t ) \|_{W^{1,q}(\TT^d)}  &\leq \nu C_R   \sigma    \mu^{1+ \frac{d-1}{p'}  - \frac{d-1}{q}}  \big|  g_{\kappa} (\sigma  t) \big|.
\end{align}
Integrating \eqref{eq:w_p_proof_4} in time and using \eqref{eq:def_g_theta_g_w} we have
\begin{align*}
\|  w_p   \|_{L^1_t W^{1,q} }  &\leq \nu C_R  \kappa^{-1}  \sigma  \mu^{1+ \frac{d-1}{p'}  - \frac{d-1}{q}}    = \nu C_R \sigma  \mu^{\frac{d-1}{p'}  - \frac{d-1}{q}} .
\end{align*}
Thanks to Lemma \ref{lemma:bound_hierarchy}, it follows from the above that
\begin{align*}
\|  w_p  \|_{L^1_t W^{1,q} }   \leq \nu C_R  \l^{- \gamma}  .
\end{align*}

\end{enumerate}
\end{proof}

\begin{lemma}[Estimate on $w_c$]\label{lemma:w_c_estimate}
There holds
\begin{align*}
\|w_c \|_{L^\infty_t L^{p'}}  & \leq C_R  \nu  \sigma^{-1} ,\\
\| w_c  \|_{L^1_t W^{1,q}} & \leq C_R   \nu \kappa^{-1}.
\end{align*}
In particular, for $\l$ sufficiently large  
\begin{align*}
\| w_c  \|_{L^\infty_t L^{p'}} & \leq \nu \| R \|_{L^1_{t,x}}^\frac{1}{q},   \\
\| w_c  \|_{L^1_t W^{1,q}} &\leq \delta/2.
\end{align*}
\end{lemma}
\begin{proof}

We first prove the $L^\infty_t L^{p'}$ estimate, and then the Sobolev estimate $L^1_t   W^{1,q}$.

\begin{enumerate}[leftmargin=*]
    
\item \underline{\textit{$L^\infty_t L^{p'}$ estimates:}}
Taking $L^{p'}$ norm in space we have
\begin{align}\label{eq:proof_w_c_1}
\| w_{c } (t) \|_{p'} & \leq   \nu  \big|  g_{\kappa} (\sigma  t) \big| \sum_{1\leq k \leq d} \Big\| \mathcal{B}  ( \nabla  B_k  , \bw_k (\sigma \cdot) ) \Big\|_{p'}  .
\end{align}
By Lemma \ref{lemma:cheapbound_B} we get
\begin{equation}
\begin{aligned}\label{eq:proof_w_c_2}
\Big\| \mathcal{B}  ( \nabla  B_k  , \bw_k (\sigma \cdot) ) \Big\|_{p'}  & \leq   C_R  \big\|    \mathcal{R} \bw_k (\sigma \cdot)   \big\|_{p'}  .
\end{aligned}
\end{equation}
Since the assumption on $p,q$ implies that for $1< p' <\infty$, we have
$$
\big\|    \mathcal{R} \bw_k (\sigma \cdot)   \big\|_{p'} \lesssim \sigma^{-1} .
$$
Then it follows from \eqref{eq:proof_w_c_1} and \eqref{eq:proof_w_c_2} that
\begin{align*} 
\| w_{c } (t) \|_{p'}  &\leq  C_R  \nu  \sigma^{-1} \big|  g_{\kappa} (\sigma  t) \big|,
\end{align*}
which implies the desired bound thanks to \eqref{eq:def_g_theta_g_w}.

\item \underline{\textit{$L^1_t W^{1,q}$ estimates:}}

We take $ W^{1,q}$ norm in space to obtain
\begin{align*} 
\| w_{c} (t) \|_{W^{1,q}} & \leq   \nu  \big|  g_{\kappa} (\sigma  t) \big| \sum_{1\leq k \leq d}  \big\| \mathcal{B}  ( \nabla  B_k  , \bw_k (\sigma \cdot) )  \big\|_{W^{1,q}}.
\end{align*}

By Poincare's inequality, we have
\begin{align}\label{eq:proof_w_c_4}
\| w_{c} (t) \|_{W^{1,q}} & \lesssim   \nu  \big|  g_{\kappa} (\sigma  t) \big| \sum_{1\leq k \leq d}  \big\| \nabla \mathcal{B}  ( \nabla  B_k  , \bw_k (\sigma \cdot) )  \big\|_{q} .
\end{align}

In fact, a slight modification of the proof of Lemma \ref{lemma:cheapbound_B} gives  
$$
\| \nabla \mathcal{B}  ( \nabla  a  ,f ) \|_{r} \lesssim \|a \|_{C^2} \Big[ \|  \mathcal{R} f \|_r +\|\nabla   \mathcal{R} f \|_r  \Big] \quad \text{for all $1 \leq r \leq \infty$ and $a ,f \in C^\infty (\TT^d)$}.
$$
Due to the assumptions on $p,q$, $1 \leq q < p' < \infty$, which in particular implies that
\begin{equation}
\begin{aligned}\label{eq:proof_w_c_5}
\big\| \nabla \mathcal{B}  ( \nabla  B_k  , \bw_k (\sigma \cdot) ) \big\|_{p'}  & \leq  C_R   ,
\end{aligned}
\end{equation}
where we used the fact that $\nabla   \mathcal{R} $ is a Calder\'on-Zygmund operator on $\TT^d$.

Combining \eqref{eq:proof_w_c_4} and \eqref{eq:proof_w_c_5} we have
\begin{align*} 
\| w_{c } (t) \|_{W^{1,q}}  \leq  C_R \nu  \big|  g_{\kappa} (\sigma  t) \big|,   
\end{align*}
which implies the desired bound after integrating in time thanks to \eqref{eq:def_g_theta_g_w}.

\end{enumerate}

\end{proof}

\subsection{Proof of the perturbation part of Proposition \ref{prop:main_prop}}

Since the second part of \eqref{eq:main_prop_4} has been proved in Lemma \ref{lemma:conclusion_new_rho_u_R}, we finish proving  \eqref{eq:main_prop_0}--\eqref{eq:main_prop_5} of Proposition \ref{prop:main_prop} in the lemma below. 

\begin{lemma}\label{lemma:proof_main_prop_1}
There exist a universal constant $M$ and a large $N \in \NN$ such that for all $\l(\nu, \delta   ,R)$ sufficiently large, the following holds.
\begin{enumerate}
    \item The density perturbation $\theta$ satisfies 
    $$ 
    \nu^{-1} \|\theta \|_{L^1_t L^p} \leq M \|R \|_{L^1_{t,x}}^{1/p} \quad \text{and } \quad \Supp \theta    \subset I_r\times \TT^d  .
    $$
    
    \item The vector field perturbation $w$ satisfies
    $$
    \nu\| w \|_{L^\infty_t L^{p'}} \leq M \|R \|_{L^1_{t,x}}^{1/p'} \quad \text{and } \quad \| w \|_{L^1 W^{1, q}} \leq \delta  .
    $$

    \item The density perturbation $\theta$ has zero mean, and for all $t\in [0,T]$ and $\varphi \in C^\infty(\TT^d)$
    $$ 
    \int_{\TT^d } \theta(t,x)  \varphi(x) \, dx \leq \delta  \| \varphi\|_{C^{N}}.
    $$
    
\end{enumerate}
\end{lemma}
\begin{proof}
By Lemmas \ref{lemma:theta_p_estimate},\ref{lemma:theta_c_estimate},\ref{lemma:theta_o_estimate} and Lemmas \ref{lemma:w_p_estimate} and \ref{lemma:w_c_estimate}, for $\l$ sufficiently large, we conclude that
\begin{align*}
\|\theta \|_{L^1_t L^p} &\lesssim  \nu  \|R \|_{L^1_{t,x}}^{1/p}  \\
\| w \|_{L^\infty_t L^{p'}}&\lesssim  \nu^{-1}  \|R \|_{L^1_{t,x}}^{1/p'}
\end{align*}
with implicit constants independent of $\l$ and $(\rho,u,R)$.
We thus choose the constant $M$ to be maximum of the two implicit constants.

To see that $\Supp \theta  \subset I_r \times \TT^d $, we simply note that by \eqref{eq:def_chi_j}, the coefficients $A_k$ and $B_k$ in the definitions of $\theta_p$, $\theta_c$ and $\theta_o$ all verify this property.

By Lemmas \ref{lemma:w_p_estimate} and \ref{lemma:w_c_estimate} again, for $\l$ sufficiently large, we have
$$
\| w \|_{L^1_t W^{1, q}} \leq \delta .
$$

Finally, let us show the last property. Noticing that $\theta_p +\theta_c$ has zero mean by default and $\theta_o$ is a divergence, we conclude that the density perturbation $\theta$ is mean-free. 
To show the last estimate, fix a test function $\varphi \in C^\infty(\TT^d)$.  By definitions, we have
\begin{align*}
\Big| \int_{\TT^d} \theta \varphi \, dx \Big| \leq \Big| \int_{\TT^d}  \theta_p   \varphi \, dx\Big|  + \Big| \int_{\TT^d}  \theta_c  \varphi \, dx\Big|+    \Big| \int_{\TT^d} \theta_o \varphi \, dx  \Big|.
\end{align*}
We show the bounds for $\theta_p$ and $\theta_o$ since the argument can be adapted to bound $ \theta_c$ as well.

On one hand, applying Lemma \ref{lemma:mean_values_oscillation} we have
\begin{align*}
\Big|  \int_{\TT^d}  \theta_p   \varphi \, dx\Big| \lesssim \sigma^{-N} \| \widetilde{g}_{\kappa}  \|_\infty \sum_{1 \leq k \leq d}\| A_k \varphi \|_{C^N}     \| \bp_k \|_2.
\end{align*}
Recall that $\gamma N > d+1 $, and then
\begin{align}
\Big|  \int_{\TT^d}  \theta_p   \varphi \, dx\Big| & \leq C_R   \l^{-d-1} \kappa  \| \bp_k \|_2 \|   \varphi \|_{C^N}  \nonumber \\
& \leq C_R   \l^{-1}     \|   \varphi \|_{C^N}  .  \label{eq:proof_main_prop_part1_1}
\end{align}
On the other hand, by Lemma \ref{lemma:theta_o_estimate}, we have
\begin{align}
\Big| \int_{\TT^d} \theta_o \varphi \, dx  \Big| &\leq \|\theta_o \|_{L^\infty_{t,x}}  \|\varphi \|_\infty \nonumber\\
& \leq C_R \sigma^{-1} \|\varphi \|_\infty. \label{eq:proof_main_prop_part1_2}
\end{align}

Putting \eqref{eq:proof_main_prop_part1_1} and \eqref{eq:proof_main_prop_part1_1} together and increasing the value of $\l$ if necessary, we obtain
$$
\Big| \int_{\TT^d} \theta  \varphi \, dx  \Big|  \leq \delta \|\varphi \|_{C^N}.
$$
\end{proof}

\section{Proof of the proposition \ref{prop:main_prop}: estimates on the new defect field}\label{section:prop_proof3}
We now turn to the final step of proving Proposition \ref{prop:main_prop}. Recall that we need to estimates the terms that solve the divergence equations
\begin{align*}
\D R_{\Osc}  & = \p_t \theta_o  + \D (\theta_p w_p) + \D R,  \\
\D R_{\Tem}  & =  \p_t \theta_p + \p_t \theta_c,  \\
\D R_{\Lin}  & = \D (  \theta u + \rho w ), \\
\D R_{\Cor}  & =  \D (    \theta  w_c + (\theta_o +  \theta_c ) w_p ) .
\end{align*}
The linear error $R_{\Lin}$ and correction error $R_{\Cor}$ can be estimated easily by standard methods. The temporal error $ R_{\Tem}$ is subtler and we need to exploit the derivative gain given by the potential $\mathbf{\Omega}_k$   in Theorem \ref{thm:WPhi_exact_stationary}. Such a difficulty is not present in \cite{MR3884855,1902.08521}.

For the oscillation error $R_{\Osc}$, we will use the decomposition done at the end of Section \ref{section:prop_proof1}, which reads
$$
R_{\Osc} =  R_{\text{osc,x}}  + R_{\text{osc,t}}    +  R_{\Rem}  .
$$
We summarize how each part of $R_{\Osc} $ will be estimated as follows.
\begin{enumerate}
    \item As typical in the literature, $R_{\text{osc,x}}$ can be shown to be small due to a gain of $\sigma^{-1}$ given by the antidivergence.
    \item The term $R_{\text{osc,t}} $ is small by itself since it is the outcome of a temporal cancellation.

    \item Finally, $R_{\Rem}$ is the leftover old defect field that is small due to our choice of cutoffs $\chi_k$ in \eqref{eq:def_chi_j}.
\end{enumerate}

\subsection{Temporal error}
\begin{lemma}[$R_{\Tem}$ estimate]\label{lemma:tem_error}
For $\l$ sufficiently large,
$$
\| R_{\Tem}  \|_{L^1_{t,x}}  \leq  \frac{\delta }{16 }  .
$$
\end{lemma}
\begin{proof}
We may rewrite it as
\begin{align*}
R_{\Tem}  &= \nu^{-1}      \sum_{1\leq k \leq d} \p_t( \widetilde{g}_{\kappa} (\sigma t))  \mathcal{B} (A_k , \bp_k(\sigma\cdot) )     +    \widetilde{g}_{\kappa} (\sigma t)    \mathcal{B} ( \p_t A_k , \bp_k(\sigma \cdot) ) \\
& := R_{\Tem,1} + R_{\Tem,2} .
\end{align*}
 
We will treat the second term $ R_{\Tem,2}$ as an error.  

\begin{enumerate}[leftmargin=*]
\item \underline{\textit{$R_{\Tem,1}$ estimate:}}

Taking $L^1$ in space, we get
\begin{align}\label{eq:R_l1}
\|  R_{\Tem,1}(t)   \|_{1} \lesssim \nu^{-1}  \sigma |\p_t  \widetilde{g}_{\kappa}( \sigma t)| \sum_k  \big\|  \mathcal{B}  \big( A_k , \bp_k(\sigma\cdot)   \big)  \big\|_1.   
\end{align}
Thanks to the potential $ \mathbf{\Omega}_k$, we have
\begin{align}\label{eq:proof_R_tem_1}
 \big\| \mathcal{B}  \big( A_k , \bp_k(\sigma\cdot)   \big)   \big\|_1 =  \big\|  \mathcal{B}  \big( A_k , \D \mathbf{\Omega}_k(\sigma \cdot )    \big)     \big\|_1.
\end{align}
Next, we  apply Lemma \ref{lemma:cheapbound_B} to obtain 
\begin{align*}
 \big\|  \mathcal{B}  \big( A_k , \D \mathbf{\Omega}_k(\sigma \cdot )    \big)    \big\|_1  &\leq C_R \big\| \mathcal{R}   \big(\D\mathbf{\Omega}_k(\sigma \cdot )  \big)\big\|_1 \\
  \text{\small (by periodic rescaling)}  \quad  &\leq C_R \sigma^{-1} \big\| \mathcal{R}    \D\mathbf{\Omega}_k(\sigma \cdot )  \big\|_1 \\
 \text{\small (by definition of $\mathcal{R})$} \quad  & \leq C_{r,R} \sigma^{-1} \|\mathbf{\Omega}_k \|_{r}  \\
\text{\small (by Proposition \ref{prop:bounds_periodic_phi_w})} \quad & \leq C_{r,R} \sigma^{-1} \mu^{-1}  \mu^{\frac{d-1}{p} -    \frac{d-1}{r} }, 
\end{align*}
for any $1< r < \infty$. Then we fix $r>1$ as in Lemma \ref{lemma:bound_hierarchy} so that
\begin{align*}
 \big\| \mathcal{B}  \big( A_k , \bp_k(\sigma\cdot)   \big)   \big\|_1 & \lesssim_r   C_R \sigma^{-1} \mu^{-1}   \l^{-\gamma}  ,
\end{align*}
which together with the bound
$$
\int_{[0,1 ]} |\p_t  \widetilde{g}_{\kappa}( \sigma t)|\, dt \lesssim \kappa  ,
$$
implies that
\begin{align*}
\|  R_{\Tem,1}    \|_{L^1_{t,x}} & \leq C_R \kappa \sigma \mu^{-1}  \sigma^{-1}    \l^{-\gamma}  \\
& \leq C_R \l^{-\gamma} ,
\end{align*}
where we have also used Lemma \ref{lemma:bound_hierarchy}.

Now for $\l$ sufficiently large, we have
$$
\|  R_{\Tem,1}    \|_{L^1_{t,x}}  \leq    \frac{\delta}{32} .
$$

\item \underline{\textit{$R_{\Tem,2}$ estimate:}}

We treat the second term $R_{\Tem,2}$ as an error and use Lemma \ref{lemma:cheapbound_B} to obtain that
\begin{align*}
  \| R_{\Tem,2} \|_1 \leq C_R | \widetilde{g}_{\kappa}( \sigma t)| \sum_k \|  \bp_k(\sigma \cdot ) \|_1.
\end{align*}
Using Proposition \ref{prop:bounds_periodic_phi_w} and \eqref{eq:def_g_theta_g_w},  integrating in time gives
\begin{align*}
 \| R_{\Tem,2} \|_{L^1_{t,x}}\leq C_R  \mu^{\frac{d-1}{p} - (d-1)}.
\end{align*}
Thanks to Lemma \ref{lemma:bound_hierarchy}, for $\l$ sufficiently large, we have
\begin{align*}
\| R_{\Tem,2} \|_{L^1_{t,x}} \leq  \frac{\delta}{32}.
\end{align*}

\end{enumerate}

\end{proof}

\subsection{Linear error}
\begin{lemma}[$R_{\Lin}$ estimate]
For $\l$ sufficiently large,
$$
\| R_{\Lin}  \|_{L^1_{t,x}}  \leq   \frac{\delta }{16 } .
$$
\end{lemma}
\begin{proof}
We start with H\"older's inequality
\begin{align*}
\big\|R_{\Lin}  \big \|_{L^1_{t,x}}   \leq\| \theta   \|_{L^1_{t,x}}    \| u \|_{L^\infty_{t,x}} +     \| \rho   \|_{L^\infty_{t,x}}    \| w \|_{L^1_{t,x}}.
\end{align*}

On one hand, by H\"older's inequality  we get 
\begin{equation}
\begin{aligned}\label{eq:theta_pc_L1}
\| \theta_p + \theta_c  \|_{L^1_{t,x}}  & \leq C_{R}    \nu^{-1}  \sum_{1\leq k \leq d}  \|     \widetilde{g}_{\kappa} (\sigma t)  A_k(t,x)  \bp_k(\sigma x)    \|_{L^1_{t,x}} \\
& \leq C_{R}    \nu^{-1}  \sum_{1\leq k \leq d}  \|     \widetilde{g}_{\kappa}\|_{L^1([0,1])}    \|  \bp_k     \|_{1}\\
&\leq C_{R}    \nu^{-1}    \mu^{\frac{d-1}{p}  -  {d-1} }  .  
\end{aligned}
\end{equation}
By definition of $\theta_o$ \eqref{eq:def_theta_o_1} we have
\begin{align}\label{eq:theta_o_L1}
\| \theta_o  \|_{L^1_{t,x}}  \lesssim C_{R} \sigma^{-1}\|h \|_{L^1([0,1])} \lesssim_{R} \sigma^{-1}.
\end{align}

On the other hand, since $1 \leq q <\infty $, by Lemma \ref{lemma:w_p_estimate} and Lemma \ref{lemma:w_c_estimate}  
\begin{equation}
\begin{aligned}\label{eq:w_pc_L1}
\| w  \|_{L^1_{t,x} }  & \leq \| w_p  \|_{L^1_{t } W^{1,q} } + \| w_c  \|_{L^1_{t } W^{1,q} }\\
&\leq C_{R}    \nu  (\sigma^{-1}  + \kappa^{-1 })   . 
\end{aligned}
\end{equation}

Combining \eqref{eq:theta_pc_L1}, \eqref{eq:theta_o_L1}, and \eqref{eq:w_pc_L1} we have
\begin{align*}
\big\|R_{\Lin}  \big \|_{L^1_{t,x}}   \leq C_{\rho,u,R,\nu}      \big(  \mu^{\frac{d-1}{p }  -  {d-1}   }    + \kappa^{-1} +\sigma^{-1} \big).  
\end{align*}

Thanks to Lemma \ref{lemma:bound_hierarchy}, for sufficiently large $\l$ we have
\begin{align}\label{eq:R_l2_L1_bounded}
\|R_{\Lin }  \|_{L^1_{t,x}} \leq   \frac{\delta }{16 } .
\end{align}

\end{proof}

\subsection{Correction error}
\begin{lemma}[$R_{\Cor}$ estimate]
For $\l$ sufficiently large,
\begin{align*}
\| R_{\Cor}  \|_{L^1_{t,x}} &\leq   \frac{\delta }{16 } .
\end{align*}
\end{lemma}
\begin{proof}
By H\"older's inequality we have
\begin{align*}
\| R_{\Cor}  \|_{L^1_{t,x}} \leq      \| \theta\|_{L^1 L^p}  \|  w_c  \|_{L^\infty L^{p'}} +( \| \theta_o \|_{L^1 L^p} +  \| \theta_c\|_{L^1 L^p} ) \| w_p \|_{L^\infty L^{p'}}.
\end{align*}

All terms have been estimated before, and by Lemma \ref{lemma:theta_p_estimate}, \ref{lemma:theta_c_estimate}, \ref{lemma:theta_o_estimate}, \ref{lemma:w_p_estimate}, \ref{lemma:w_c_estimate} we have
\begin{align*}
\| R_{\Cor}  \|_{L^1_{t,x}} &\lesssim_{R}       \nu^{-1} \|  w_c  \|_{L^\infty L^{p'}} + \nu( \| \theta_o \|_{L^1 L^p} +  \| \theta_c\|_{L^1 L^p} )   \\
& \leq C_{R}   \sigma^{-1}  (   \nu^{-1}   +  \nu     ), 
\end{align*}
which concludes the proof.

\end{proof}

\subsection{Oscillation errors}

We will estimate $R_{\Osc}$ according to the decomposition in Lemma \ref{lemma:final_decomposition_R_osc}.

For reference, we recall that
$$
R_{\Osc} = R_{\text{osc,x}} + R_{\text{osc,t}}  + R_{\Rem},
$$
where $R_{\text{osc,x}} $ is the error of high frequency in space
$$
R_{\text{osc,x}} =  \widetilde{g}_{\kappa}  (\sigma t)  g_{\kappa} (\sigma  t)  \sum_{1\leq k \leq d} \mathcal{B} \Big(  \nabla(A_k B_k) , \big(  \bp_k \bw_k (\sigma x) -\fint_{\TT^d}  \bp_k \bw_k \, dx \big)  \Big)  ,
$$
$R_{\text{osc,t}} $ is the error of high frequency in time
$$
R_{\text{osc,t}} =  \sigma^{-1} h(\sigma t) \mathcal{R}  \sum_{1\leq k \leq d} \ek \cdot  \p_t \nabla (\chi_k^2 R_k) ,
$$
and $R_{\Rem}$ is the remainder error
$$
R_{\Rem}= \sum_{1\leq k \leq d}  (1-\chi_k^2 ) R_k \ek.
$$

We start with $R_{\text{osc,x}} $.

\begin{lemma}[$ R_{\text{osc,x}} $ estimate]
For $\l$ sufficiently large,
\begin{align*}
\| R_{\text{osc,x}}  \|_{L^1_{t,x}} &\leq   \frac{\delta }{16 }  .
\end{align*}
\end{lemma}
\begin{proof}
Denote $\mathbf{\Theta }_k \in C^\infty_0(\TT^d) $ by
$$
\mathbf{\Theta }_k =   \bp_k \bw_k  -\fint_{\TT^d}  \bp_k \bw_k \, dx,
$$
so that $R_{\text{osc,x}} $ reads
\begin{align*}
R_{\text{osc,x}} &=  \widetilde{g}_{\kappa}  (\sigma t)  g_{\kappa} (\sigma  t)   \sum_{1\leq k \leq d}   \mathcal{B} \Big( \nabla(A_k B_k) ,  \mathbf{\Theta }_k(\sigma \cdot)  \Big) .
\end{align*}

We take $L^1 $ norm in space to obtain
\begin{align*}
\|    R_{\text{osc,x}} (t)\|_{1 } \leq \big|   \widetilde{g}_{\kappa}  (\sigma t)  g_{\kappa} (\sigma  t) \big| \sum_{1\leq k \leq d}  \big\| \mathcal{B} \big( \nabla(A_k B_k) ,  \mathbf{\Theta }_k(\sigma \cdot)  \big)    \big\|_1.
\end{align*}
Applying Lemma \ref{lemma:cheapbound_B} gives
\begin{align*}
\Big\| \mathcal{B} \big( \nabla(A_k B_k) ,  \mathbf{\Theta }_k(\sigma \cdot)  \big)  \Big\|_1  &\lesssim C_R \sigma^{-1} \| \mathcal{R}   \mathbf{\Theta }_k  \|_1  .
\end{align*}
It follows that
\begin{align*}
\|    R_{\text{osc,x}} (t)\|_{1} \leq C_R \big|   \widetilde{g}_{\kappa}  (\sigma t)  g_{\kappa} (\sigma  t) \big|  \sigma^{-1}   .
\end{align*}
So for $L^1_{t,x}$ norm we have
\begin{align*}
\|    R_{\text{osc,x}} \|_{L^1_{t,x} } \leq   C_R   \sigma^{-1} .
\end{align*}

\end{proof}

\begin{lemma}[$ R_{\text{osc,t}} $ estimate]
For $\l$ sufficiently large,
\begin{align*}
\| R_{\text{osc,t}}  \|_{L^1_{t,x}} &\leq  \frac{\delta }{16 }  .
\end{align*}
\end{lemma}
\begin{proof}
By Lemma \ref{lemma:antidivergence_bounded}, we have
$$
\| R_{\text{osc,t}} \|_{L^1_{t,x}} \lesssim  \sigma^{-1} \big\|  h(\sigma t)    \sum_{1\leq k \leq d} \ek \cdot  \p_t \nabla (\chi_k^2 R_k)  \big\|_{L^1_{t,x}}.
$$
It follows from H\"older's inequality that
$$
\| R_{\text{osc,t}} \|_{L^1_{t,x}} \leq C_R \sigma^{-1}  \big\|  h(\sigma t)   \big\|_{L^1([0,1])} \leq  C_R \sigma^{-1} .
$$

\end{proof}

\begin{lemma}[$ R_{\Rem}  $ estimate]
There holds
\begin{align*}
\| R_{\Rem}  \|_{L^1_{t,x}} &\leq \frac{\delta}{2}  .
\end{align*}
\end{lemma}
\begin{proof}
We need to estimate
\begin{align*}
\| R_{\Rem} \|_{L^1_{t,x}}  \leq \sum_{1\leq k \leq d} \| (1-\chi_k^2 ) R_k \ek \|_{L^1_{t,x}}.
\end{align*}
Note that 
$$
(t,x) \in \Supp (1-\chi_k^2 ) \Rightarrow |R_k| \leq \frac{\delta}{ 4  d} \quad \text{or } \quad t\in I_{r }^c,
$$
and thus by \eqref{eq:def_constant_r} we have
\begin{align*}
\| R_{\Rem} \|_{L^1_{t,x}}  & \leq \sum_{1\leq k \leq d} \int_{|R_k| \leq \frac{\delta}{ 4 d} } (1-\chi_k^2 ) |R_k|  \, dx dt  + \int_{t\in I_{r }^c } (1-\chi_k^2 ) |R_k|  \, dx dt \\
& \leq d \times \big(  |[0,1] \times \TT^d|  \times \frac{\delta}{4d} + 2 r   \times \|R\|_{L^\infty_{t,x}}\big)  = \frac{\delta}{2} .
\end{align*}

\end{proof}

\subsection{Conclusion of the proof of Proposition \ref{prop:main_prop}}

We can finish the proof of Proposition \ref{prop:main_prop} by showing \eqref{eq:main_prop_5}.

We take $\l$ sufficiently large so that all lemmas in this section and Lemma \ref{lemma:proof_main_prop_1} hold. Then the new defect field $R$ satisfies
\begin{align*}
\|R \|_{L^1([0,1 \times \TT^d])} & \leq \|R_{\Tem} \|_{L^1_{t,x}} +\|R_{\Lin} \|_{L^1_{t,x}} + \|R_{\Cor} \|_{L^1_{t,x}}  \\
& \qquad + \|R_{\text{osc,x}}  \|_{L^1_{t,x}} + \|R_{\text{osc,t}} \|_{L^1_{t,x}}   +  \|R_{\Rem }  \|_{L^1_{t,x}}\\ 
& \leq  5 \times \frac{\delta}{16} + \frac{\delta}{2}\\
&\leq \delta.
\end{align*}

\appendix

\section*{Acknowledgement}
AC was partially supported by the NSF grant DMS--1909849. The authors thank the anonymous referee for helpful comments.

\bibliographystyle{alpha}
\bibliography{nonuniqueness_transport} 

\end{document}